\documentclass[12pt,draftcls,onecolumn]{article}

\usepackage{amsmath,amsthm}
\usepackage[pdftex]{graphicx,color}
\usepackage{wrapfig}
\usepackage{subfigure}
\usepackage{dsfont}
\usepackage{multirow}
\usepackage{multicol}
\usepackage{algorithmic}
\usepackage[square, sort, numbers]{natbib}\citeindextrue
\usepackage[margin=1in]{geometry}

\newtheorem{theorem}{Theorem}

\newtheorem{corollary}[theorem]{Corollary}

\newtheorem{proposition}{Proposition}

\newtheorem{lemma}{Lemma}

\newcommand{\Exp}{\mbox{${\sf I} \! {\sf E}$}}

\newcommand{\Rp}{\mathds{R}_+}
\newcommand{\R}{\mbox{${\sf I} \! {\sf R}$}}
\newcommand{\floor}[1]{\lfloor#1\rfloor}

\newcommand{\dvp}{\frac{d^+}{dn}\vbar}

\def \vhat{\hat v}
\def \vbar{\bar v}
\def \Vbar{\bar V}
\def \chat{\hat c}
\def \cbar{\bar c}

\begin{document}
%
\title{A New Optimal Stepsize For Approximate Dynamic Programming}

\author{Ilya~O.~Ryzhov\thanks{Ilya O. Ryzhov is with the Robert H. Smith School of Business, University of Maryland, College Park, MD 20742 USA e-mail: iryzhov@rhsmith.umd.edu.}
\and 
Peter I. Frazier\thanks{Peter I. Frazier is with the Department of Operations Research and Information Engineering, Cornell University, Ithaca, NY 14853 USA.}
\and Warren B. Powell\thanks{Warren B. Powell is with the Department of Operations Research and Financial Engineering, Princeton University, Princeton, NJ 08544 USA.}%
}

\maketitle

\begin{abstract}
Approximate dynamic programming (ADP) has proven itself in a wide range of applications spanning large-scale transportation problems, health care, revenue management, and energy systems. The design of effective ADP algorithms has many dimensions, but one crucial factor is the stepsize rule used to update a value function approximation. Many operations research applications are computationally intensive, and it is important to obtain good results quickly. Furthermore, the most popular stepsize formulas use tunable parameters and can produce very poor results if tuned improperly. We derive a new stepsize rule that optimizes the prediction error in order to improve the short-term performance of an ADP algorithm. With only one, relatively insensitive tunable parameter, the new rule adapts to the level of noise in the problem and produces faster convergence in numerical experiments.
\end{abstract}

\section{Introduction}

Approximate dynamic programming (ADP) has emerged as a powerful tool for solving stochastic optimization problems in inventory control \citep{AdKl12}, emergency response \citep{MaReHeTo10}, health care \citep{HeZhPo12}, energy storage \citep{LaMaSe10,LoMi10,Se10}, revenue management \citep{ZhAd09}, and sensor management \citep{Ca97}.  In recent research, ADP has been used to solve a large-scale fleet management problem with 50,000 variables per time period and millions of dimensions in the state variable \citep{SiDaGeGiNiPo09}, and an energy resource planning problem with 175,000 time periods \citep{PoGeLa12}. Applications in operations research are especially demanding, often requiring the sequential solution of linear, nonlinear or integer programming problems. When an ADP algorithm is limited to a few hundred iterations, it is important to find a good solution as quickly as possible, a process that hinges on a stepsize (or learning rate) which controls how new information is merged with existing estimates.


We illustrate the learning process using the language of classical Markov decision processes. Consider an infinite-horizon dynamic program where $V\left(S\right)$ is the value of being in state $S\in\mathcal{S}$ and $C(S,x,W)$ is a possibly random reward earned from being in state $S$, taking action $x \in \mathcal{X}$ and then observing random information $W$. It is well-known \citep{Ho60, Pu94} that we can find the optimal, infinite-horizon value of each state using value iteration, which requires iteratively computing, for each state $S\in\mathcal{S}$,
\begin{equation}
V^n(S) = \max_{x\in\mathcal{X}} \Exp\left[ C(S,x,W) + \gamma  V^{n-1}(S'(S,x,W))|S,x \right] \label{eq:valueiteration}
\end{equation}
where $\gamma < 1$ is a discount factor and $S'$ is a random variable describing the next state given that we were in state $S$, took action $x$, and observed $W \in \mathcal{W}$.


There are many problems where \eqref{eq:valueiteration} is difficult to solve due to the curse of dimensionality. For this reason, there is a tradition, dating back to Bellman's earliest work \citep{BeDr59}, of solving this equation approximately.  This field has evolved under a variety of names including approximate dynamic programming (ADP), neuro-dynamic programming, and reinforcement learning \citep[see][]{BeTs96, SuBa98, SiBaPo04, Po11}. In one major class of ADP methods known as \textit{approximate value iteration}, an observation of the value $V\left(S\right)$ is bootstrapped from an approximation of the downstream value of $S'$, and then used to update that approximation. A generic procedure for computing the observation is given by
\begin{equation}
\vhat^n = \max_{x^n} \sum_{w\in \mathcal{W}} P\left(W^{n+1} = w\,|\,S^n,x^n\right)\left[C\left(S^n,x^n,w\right) + \gamma\Vbar^{n-1}\left(S^{n+1}\left(S^n,x^n,w\right)\right)\right]\mbox{,}\label{eq:basicadp}
\end{equation}
where $\Vbar^{n-1}$ is the value function approximation, and $S^n \in \mathcal{S}$ is our state during the $n$th iteration of the ADP algorithm. We intend (\ref{eq:basicadp}) only to illustrate the concept of constructing $\vhat^n$ from $\Vbar^{n-1}$; in practice, the summation in (\ref{eq:basicadp}) is also approximated. The expectation within the max operator can be avoided using the concept of the post-decision state \citep{VaBeLeTs97,Po11}, which enables us to compute a modified version of (\ref{eq:basicadp}) exceptionally quickly. This makes approximate value iteration particularly useful for online applications (that is, those run in the field), since it is very easy to implement.


Regardless of the particular technique used, we update $\Vbar^{n-1}$ by smoothing it with the new observation $\vhat^n$, obtaining
\begin{equation}
\Vbar^n\left(S^n\right) = \left(1-\alpha_{n-1}\right)\Vbar^{n-1}\left(S^n\right) + \alpha_{n-1}\vhat^n\label{eq:basicsmoothing}
\end{equation}
where $0 < \alpha_{n-1} \leq 1$ is a stepsize (or learning rate). Note again that (\ref{eq:basicsmoothing}) uses statistical bootstrapping, where the estimate of the value $\vhat^n$ depends on a statistical approximation $\Vbar^{n-1}\left(S^n\right)$. This is the defining characteristic of approximate value iteration, which has proven to be very successful in broad classes of operations research applications. The reinforcement learning community uses a closely related algorithm known as Q-learning \citep{watkins1992}, which uses a similar bootstrapping scheme to learn the value of a state-action pair. In both approximate value iteration and Q-learning, the stepsize plays two roles. First, it smooths out the effects of noise in our observations (the lower the stepsize, the smoother the approximation). Second, it determines how much weight is placed on new rewards (the higher the stepsize, the more a new reward is worth). This dual role of the stepsize is specific to bootstrapping-based methods, whose ease of use makes them a natural approach for large-scale applications in operations research where rate of convergence is crucial. See e.g. \cite{ToPo06} or ch. 14 of \cite{Po11} for more examples of such applications. The method of using a stepsize to update the value of a state-action pair is based on the field of stochastic approximation; see \cite{Wa69} and \cite{KuYi97} for thorough treatments of this field. \cite{Ts94} and \cite{JaJoSi94} were the first to apply this theory to show the convergence of an ADP algorithm when the stepsize rule satisfies certain conditions.

In general, ADP practice shows a strong bias toward simple rules that are easy to code. One example of such a stepsize rule is $\alpha_{n-1}=1/n$, which has the effect of averaging over the observations $\vhat^n$. In fact, this rule satisfies the necessary theoretical conditions for convergence, which has made it into a kind of default rule \citep[see e.g.][for a recent example]{azar2011}. However, the literature has acknowledged \citep{Sz97,EvMa03,Go08} that the $1/n$ rule can produce very slow convergence; one of our contributions in this paper is to derive new theoretical bounds providing insights into the weakness of this rule. For this reason, many practitioners use a simple constant stepsize such as $\alpha_{n-1} \equiv 0.1$. Nonetheless, the constant stepsize can produce slow initial convergence for some problems and volatile, non-convergent estimates in the limit. It is also easy to construct problems where any single constant will work poorly. \cite{simao2009} solves an inventory problem for spare parts, where a high-volume spare part may remain in inventory for just a few days, while a low-volume part may remain in inventory for hundreds of days. A small stepsize will work very poorly with a low-volume part, while large stepsizes fail to dampen the noise, and are not appropriate for high-volume parts.

Such applications require the use of stochastic stepsize rules, where $\alpha_{n-1}$ is computed adaptively from the error in the previous prediction or estimate. These methods include the stochastic gradient rule of \cite{BeMePr90} (and other stochastic gradient algorithms, e.g. by \cite{DuHaSi11} and \cite{ScZhLe13}), the Delta-Bar-Delta rule of \cite{Sutton92} and its variants \citep{MaSuDePi12}, and the Kalman filter \citep{Sten94, ChVr06}. A detailed survey of both deterministic and stochastic stepsizes is given in \cite{GePo06}, with additional references in \cite{MaSuDePi12}. The main challenge faced by these methods is that the prediction error is difficult to estimate in a general MDP, often resulting in highly volatile stepsizes with large numbers of tunable parameters. A recent work by \cite{HuLe07} adopts a different approach based on the relative frequency of visits to different states, but is heavily tied to on-policy learning, whereas practical implementations often use off-policy learning to promote exploration \citep{SuSzMa08}. In all of these cases, the literature largely ignores the dependence of the observation $\vhat^n$ on the previous value function approximation $\Vbar^{n-1}$, arguably the defining feature of approximate value iteration. For instance, the OSA algorithm of \cite{GePo06}, which can be viewed as a bias-adjusted Kalman filter \cite[or BAKF, the name used in][]{Po11}, assumes independent observations.

We approach the problem of stepsize selection by studying an MDP with a single state and action. This model radically streamlines the behaviour of a general DP, but retains key features of DP problems that are crucial to stepsize performance, namely the bias-variance tradeoff and the dependence of observations. We use this model to make the following contributions: 1) We derive easily computable, convergent upper and lower bounds on the time required for convergence under $1/n$, demonstrating that the rate of convergence of $1/n$ can be so slow that this rule should almost never be used for ADP. 2) We derive a closed-form, easily computable stepsize rule that is optimal for the single-state, single-action problem. This is the first stepsize rule to account for the dependence of observations in ADP. The formula requires no tuning, and is easy to apply to a general multi-state problem. 3) We analyze the convergence properties of our stepsize rule. We show that it does not stall, and declines to zero in the limit. This is the first optimal stepsize for ADP that provably has these properties. 4) We present numerical comparisons to other stepsizes in a general ADP setting and demonstrate that, while popular competing strategies are sensitive to tunable parameters, our new rule is robust and fairly insensitive to its single parameter. This last property is of vital practical importance, allowing developers to focus on approximation strategies without the concern that poor performance may be due to a poorly tuned stepsize formula.

Section \ref{sec:setup} defines the optimality of a stepsize, and illustrates the need for an optimal stepsize rule by theoretically demonstrating the poor performance of $\alpha_{n-1}=1/n$ on our single-state, single-action problem. Section \ref{sec:opt} derives the optimal stepsize rule for the approximate value iteration problem, and shows how it can be used in a more general ADP setting. Section \ref{sec:mainexp} presents a numerical sensitivity analysis of the new rule in the single-state problem. Finally, Sections \ref{sec:ext}-\ref{sec:storage} present numerical results for more general ADP examples.

\section{Setup and motivation}\label{sec:setup}

Section \ref{sec:model} lays out the stylized ADP model used for our analysis, and defines the optimality of a stepsize in this setting. Section \ref{sec:slow} motivates the need for an optimal stepsize by showing that the commonly used stepsize $\alpha_{n-1}=1/n$ produces unusably slow convergence in our model.

\subsection{Mathematical model}\label{sec:model}

In the dynamic programming literature, the notion of an ``optimal'' stepsize most commonly refers to the solution to the optimization problem
\begin{equation}\label{eq:predictionerror}
\min_{\alpha_{n-1}\in \left[0,1\right]} \Exp\left[ \left(\Exp\hat{v}^n - \Vbar^{n}\left(S^n\right)\right)^2\right].
\end{equation}
We refer to the quantity inside the expectation as the \textit{prediction error}. Recall from (\ref{eq:basicadp}) that $\vhat^n$ serves as an observation (albeit an approximate one) of the value of being in state $S^n$. The prediction error is the squared difference between this observation and the current estimate $\Vbar^{n-1}\left(S^n\right)$ of the value.

The prediction error is a standard objective for an optimal stepsize rule, and is used in reinforcement learning (e.g. the IDBD algorithm of \citep{Sutton92}, used e.g. by \cite{Silver13} in RL), stochastic gradient methods \citep{BeMePr90}, Kalman filtering \citep{ChVr06}, and signal processing \citep{GePo06}. The main challenge faced by researchers is that, for a general dynamic program, (\ref{eq:predictionerror}) cannot be solved in closed form. For this reason, most error-minimizing stepsize algorithms \citep[including very recent work in this area; see][for an overview]{MaSuDePi12} adopt a gradient descent approach, in which the stepsize is adjusted based on an estimate of the derivative of (\ref{eq:predictionerror}) with respect to $\alpha_{n-1}$. The resulting stepsize algorithms are no longer optimal, and can exhibit volatile behaviour in the early stages. Many of them require extensive tuning.

While we also seek to minimize prediction error, we adopt a different approach. Instead of approximating (\ref{eq:predictionerror}) in the general case, we consider a stylized dynamic program with a single state and a single action, where (\ref{eq:predictionerror}) has a closed-form solution. In this setting, (\ref{eq:valueiteration}) reduces to $v^*=c+\gamma v^*$ and has the solution $v^* = \frac{c}{1-\gamma}$. The ADP equations (\ref{eq:basicadp}) and (\ref{eq:basicsmoothing}) reduce to
\begin{eqnarray}
\vhat^{n} &=& \chat^{n} + \gamma\vbar^{n-1}\mbox{,}\label{eq:vhat}\\
\vbar^n &=& \left(1-\alpha_{n-1}\right)\vbar^{n-1} + \alpha_{n-1}\vhat^{n}\mbox{,}\label{eq:vbar}
\end{eqnarray}
where the random variables $\chat^n$, $n = 1,2,...$ are independent and identically distributed, with $c = \Exp\chat^n$ and $\sigma^2 = Var\left(\chat^n\right)$. The prediction error in this setting reduces to the formulation
\begin{equation*}
\min_{\alpha_{n-1} \in \left[0,1\right]} \Exp\left[\left(\Exp \vhat^n - \vbar^n\right)^2\right]\mbox{.}
\end{equation*}

Although the system described by (\ref{eq:vhat}) and (\ref{eq:vbar}) is much simpler than a general MDP, it nonetheless retains two key features that are fundamental to all DPs:
\begin{itemize}
\item[1)] A tradeoff between bias and variance in the approximation $\vbar^n$, governed by the stepsize $\alpha_{n-1}$;
\item[2)] Dependence of the bootstrapped observation $\vhat^n$ on the approximation $\vbar^{n-1}$.
\end{itemize}
In Section \ref{sec:finitehorizon}, we consider a finite-horizon extension that captures a third key feature:
\begin{itemize}
\item[3)] Time-dependence of the bias-variance tradeoff.
\end{itemize}
Of course, in a general DP, these issues exhibit much more complex behaviour than in the streamlined single-state, single-action model. However, the stylized model is still subject to these issues, and can provide insight into how they can be resolved in the general case. The main advantage offered by this model is that it allows us to address these issues using a closed-form solution for the optimal stepsize, explicitly capturing the relationship between the bias-variance tradeoff and the dependence of the observations. We will then be able to adapt the solution of the single-state problem to general dynamic programs (in Section \ref{sec:unknownc}).

We briefly note that we allow the observation $\chat^n$ in (\ref{eq:vhat}) to be random. At a very high level, this allows us to view the single-state, single-action problem as a stand-in for an infinite-horizon MDP in steady state. Recall the well-known property of Markov decision processes that a) the policy produced by the basic value iteration update in (\ref{eq:valueiteration}) converges to an optimal policy and b) the probability that we are in some state $s$ converges to a steady-state distribution \citep[see][]{Pu94}. As a result, the unconditional expectation of the contribution earned at each iteration approaches a constant that we can denote by $c$. Again, while the single-state, single-action model cannot capture all of the complexity of a general DP, it allows us to distill a large class of DPs into a simple and elegant archetype capturing key behaviours common to that class.

\subsection{Motivation: slow convergence of $\alpha_{n-1}=1/n$}\label{sec:slow}

The research on error-minimizing stepsizes is motivated by the poor practical performance of simple stepsize rules. Among these, the most notable is $\alpha_{n-1} = 1/n$, which produces provably convergent estimates of the value function \citep{Ts94}, and thus persists in the literature as a kind of default rule, as evidenced by its recent use in e.g. \cite{azar2011}. The theoretical worst-case convergence rate of this stepsize is known to be slow \citep{EvMa03}. We now derive new bounds that are easier to compute and demonstrate that the $1/n$ rule is unusably slow even for the stylized single-state, single-action model of Section \ref{sec:model}.

Consider the ADP model of (\ref{eq:vhat}) and (\ref{eq:vbar}). For simplicity, we assume in this discussion that $\chat^n = c$ for all $n$, that is, all the rewards are deterministic. If an algorithm performs badly in this deterministic case, we generally expect it to perform even worse when $\chat^n$ is allowed to be random, since increasing noise generally slows convergence. We briefly summarize our results and give a numerical illustration; the full technical details can be found in the Appendix.

\begin{theorem}\label{eq:LowerBound0}
$\vbar^n \ge \frac{c}{1-\gamma}\left( 1 - (n+1)^{-(1-\gamma)}\right)$ for $n = 0,1,2,...$.
\end{theorem}

\begin{theorem}\label{cor:ChooseOne0}
$\vbar^n \le \frac{c}{1-\gamma}\left[ 1 - bn^{-(1-\gamma)} - \frac{1-\gamma}{\gamma}\frac{1}{n} \right]$
for all $n = 1,2,...$ where $b = \frac{\gamma^2 + \gamma - 1}{\gamma}$.
\end{theorem}

In our numerical illustration, we fix $c$ to 1, because it only enters as a multiplicative factor in the bounds and in the true value function as well.  Thus $\gamma$ is our only free parameter.  The results are plotted on a log-scale in Figure \ref{fig:results}. As $n$ grows large the upper and lower bounds both approach the limiting value $v^* = 1/(1-\gamma)$. Convergence slows as $\gamma$ increases.

\begin{figure}[t]
\centering
\subfigure[$\gamma=0.7$.]{
\includegraphics[width=3in]{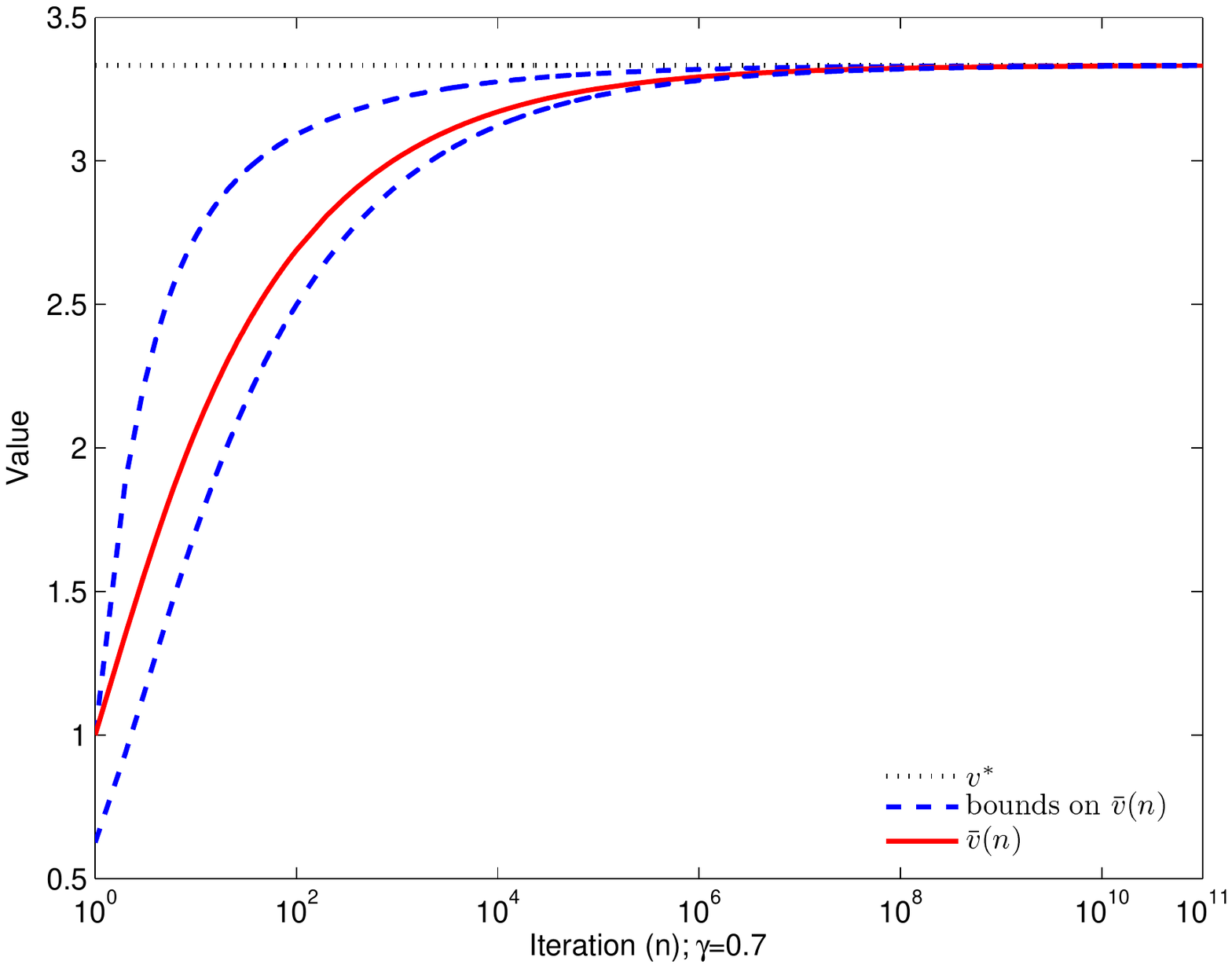}
}
\subfigure[$\gamma=0.8$.]{
\includegraphics[width=3in]{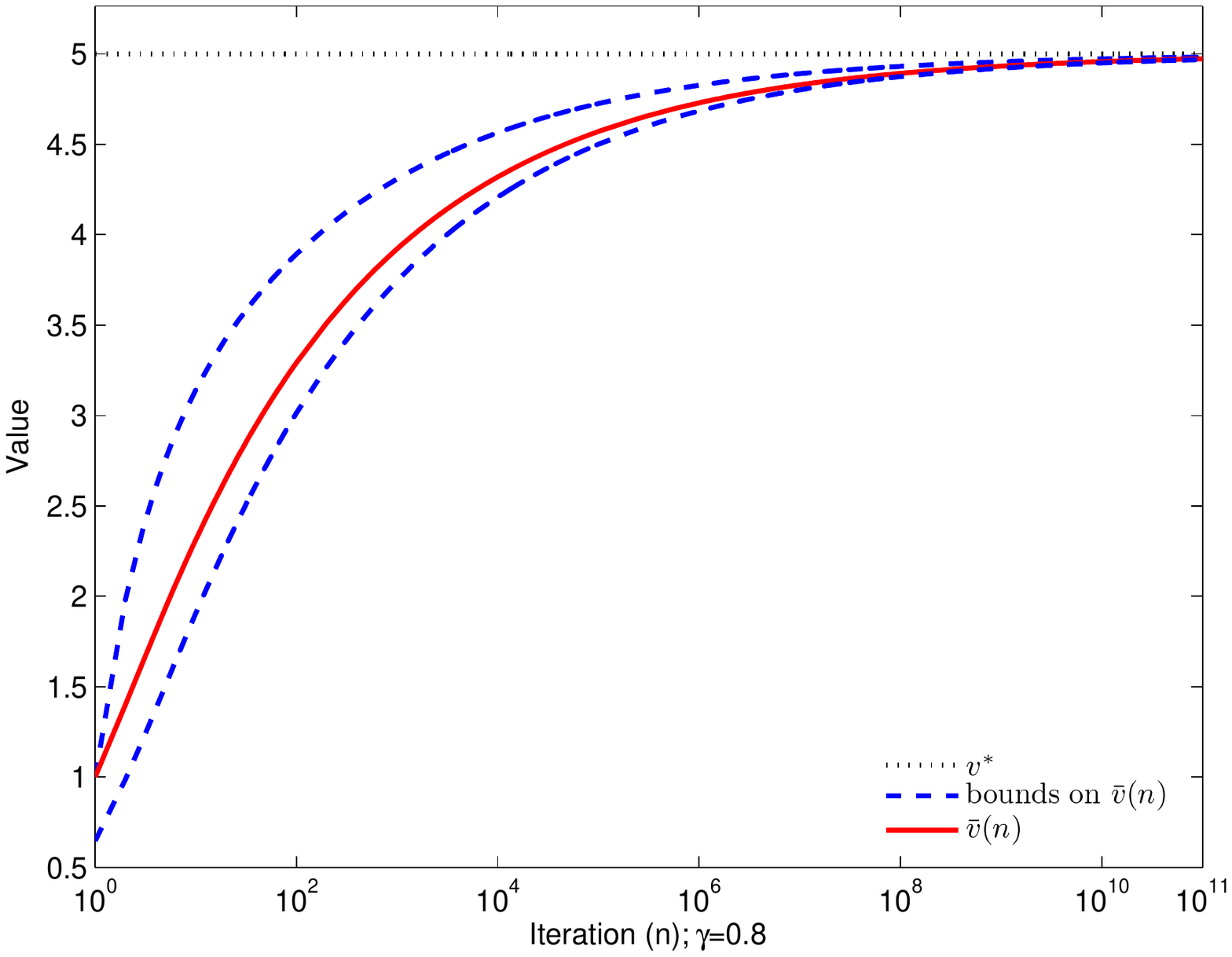}
}
\subfigure[$\gamma=0.9$.]{
\includegraphics[width=3in]{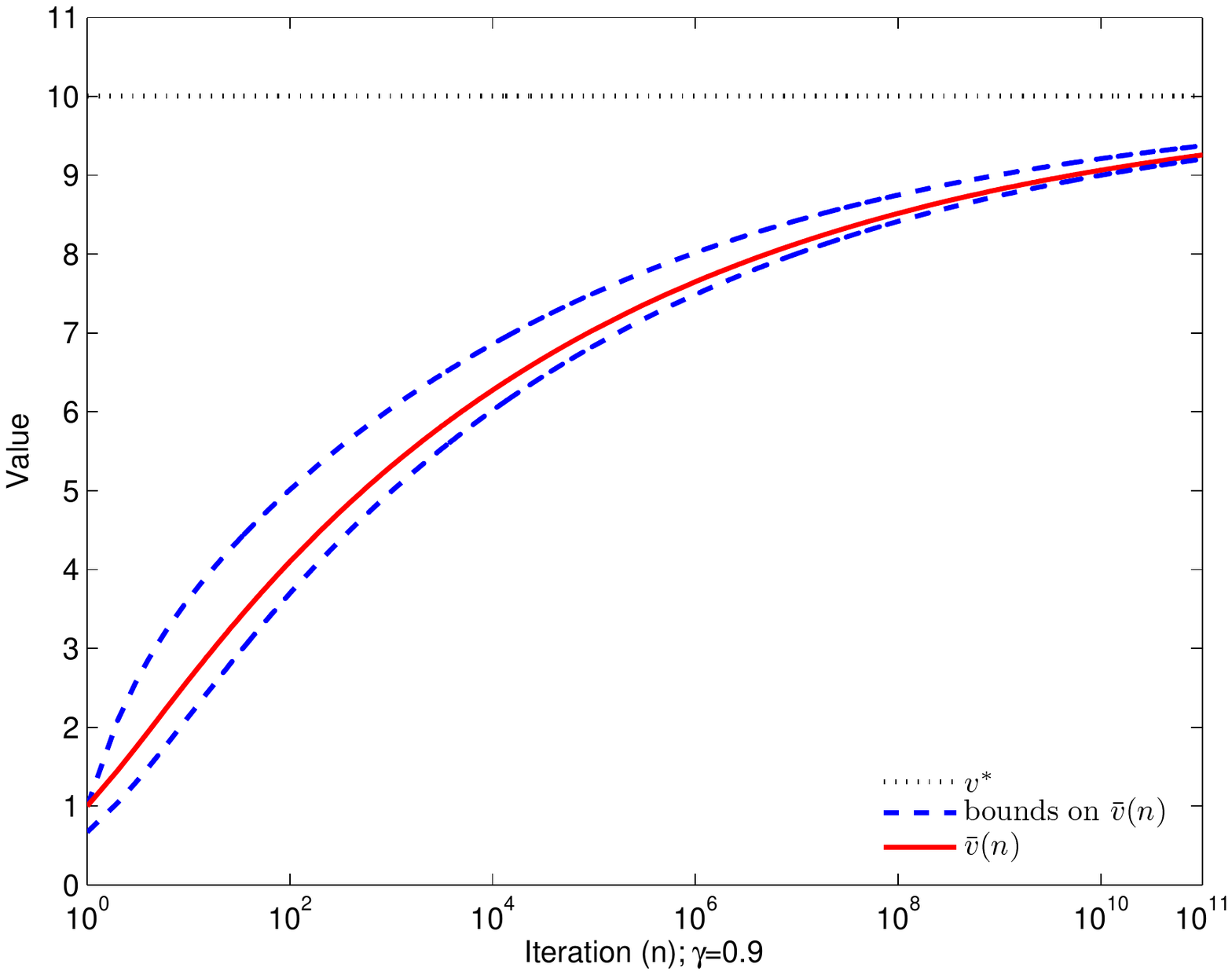}
}
\subfigure[$\gamma=0.95$.]{
\includegraphics[width=3in]{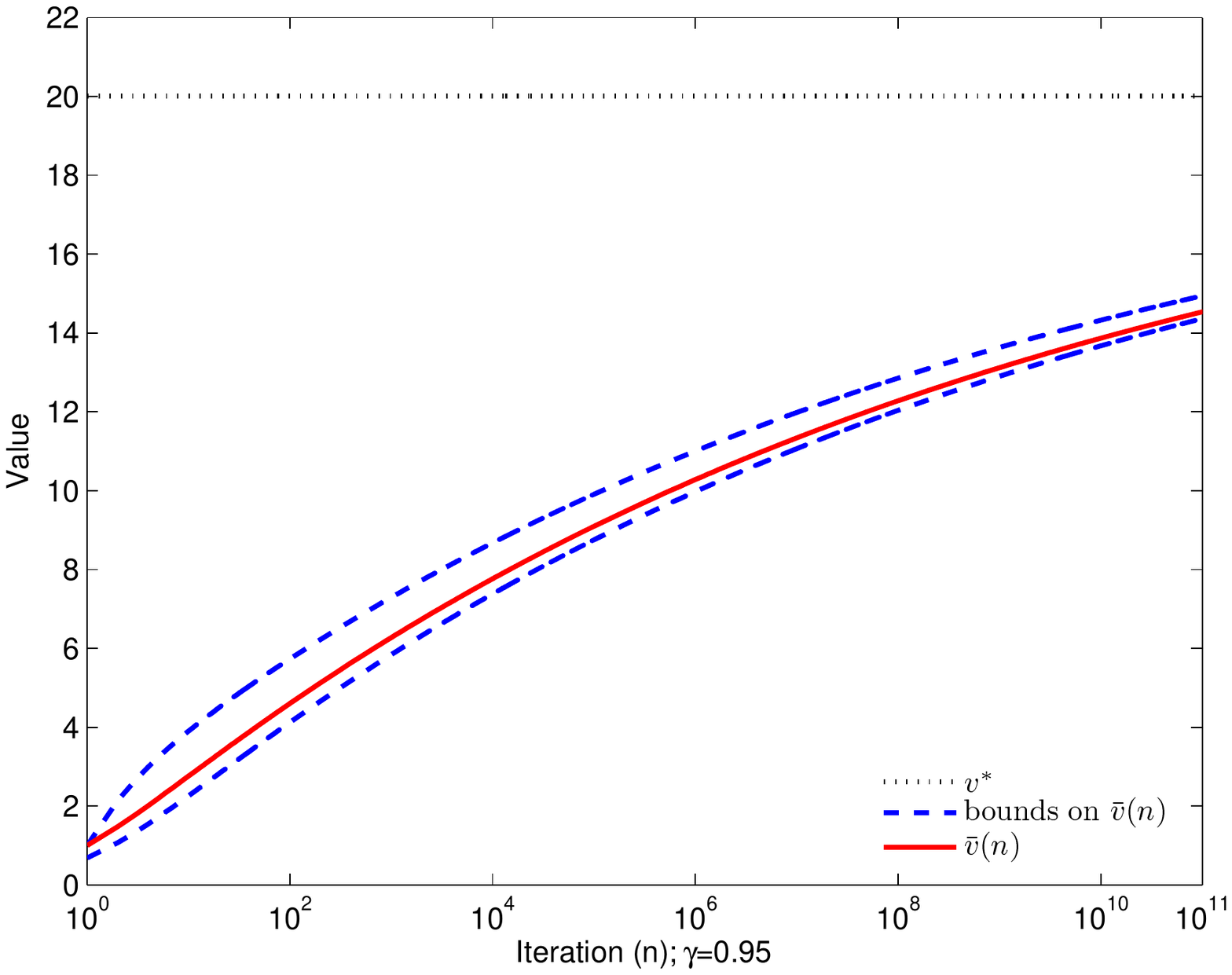}
}
\caption{$\vbar(n)$ and its upper and lower bounds for different discount factors.}\label{fig:results}
\end{figure}

\begin{figure}[t]
\centering
\subfigure[$0.65 \leq \gamma \leq 1$.]{
\includegraphics[width=3in]{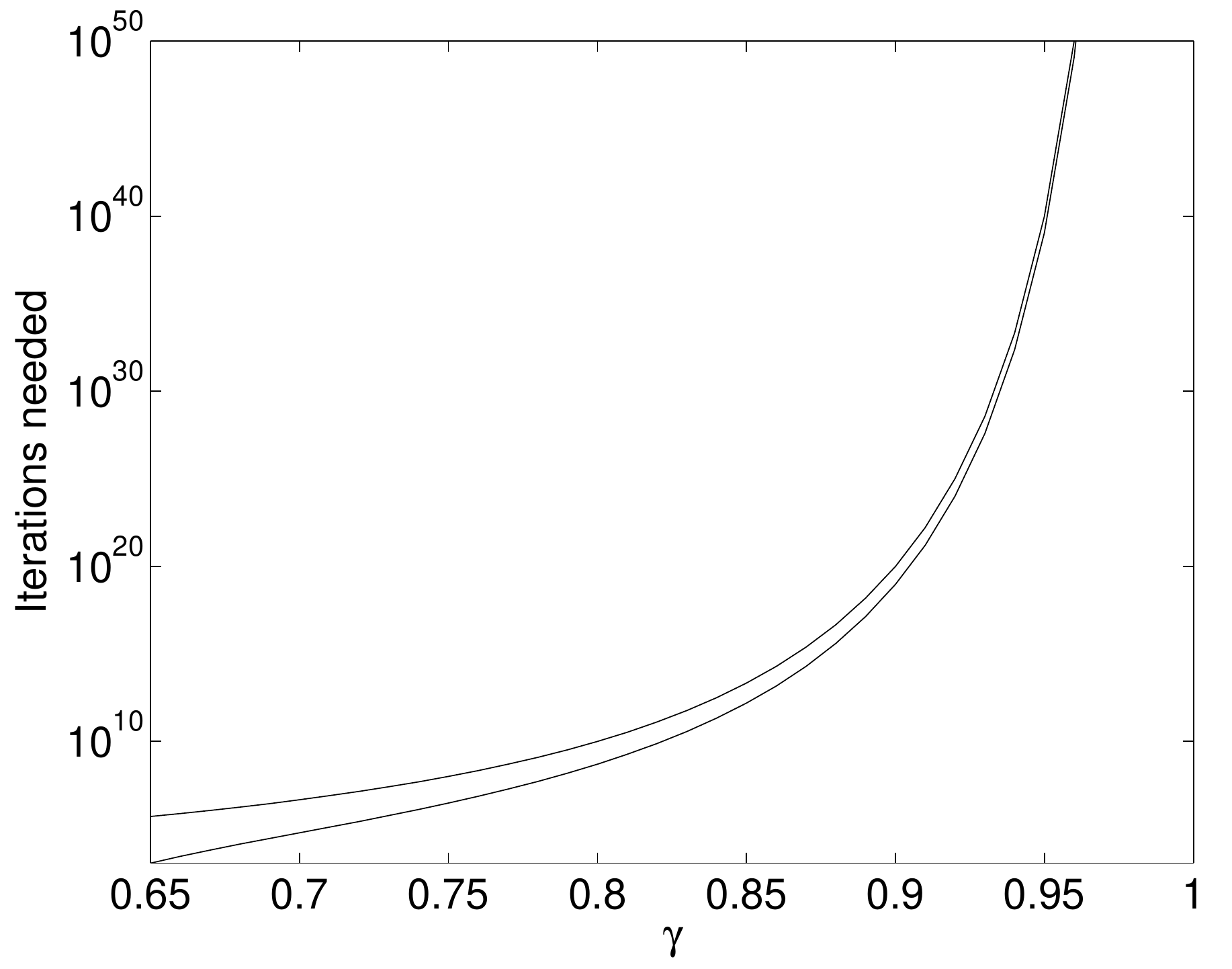}
}
\subfigure[$0.65 \leq \gamma \leq 0.99$.]{
\includegraphics[width=3in]{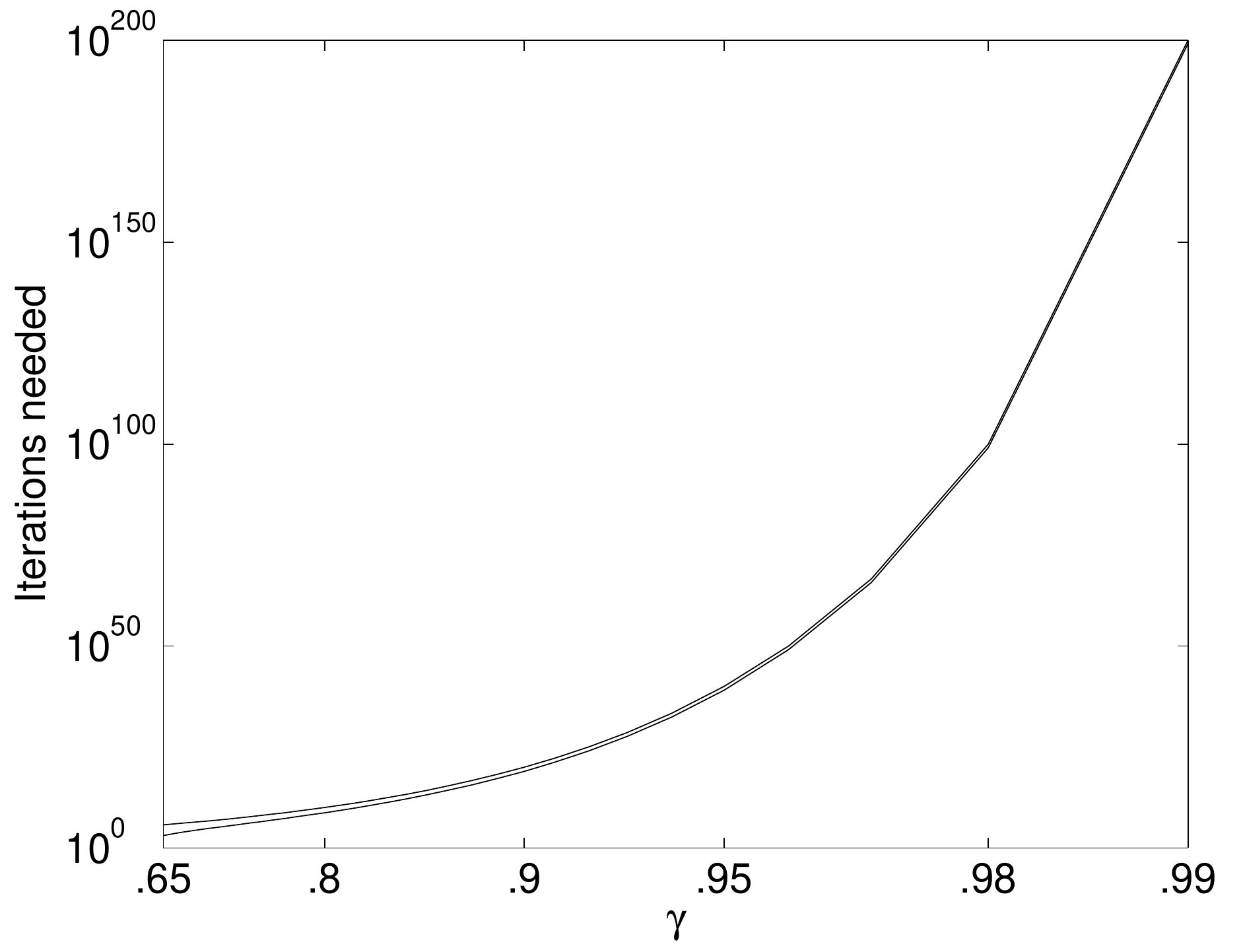}
}
\caption{Upper and lower bounds on number of iterations needed to get within 1\% of optimal, plotted for different ranges of $\gamma$.}\label{fig:IterationsNeeded}
\end{figure}

In Figure \ref{fig:IterationsNeeded}, we show the number of iterations before $\vbar^n$ reaches 1\% of optimal.  The lower bound on the value of $\vbar^n$ gives an upper bound on the number of iterations needed, and the upper bound on $\vbar^n$ gives a lower bound on the iterations needed.  For $\gamma$ near $.7$, we already require $10,000$ iterations, causing difficulty for applications requiring a significant amount of time per iteration.  Then, as $\gamma$ grows larger than $.8$ we require at least $10^8$ iterations, which is impractical for almost any application.  As $\gamma$ grows above $.9$, the number of iterations needed is at least $10^{19}$.

We see that, in this simple problem, approximate value iteration with stepsize $1/n$ converges so slowly as to be impractical for most infinite horizon applications, particularly when the discount factor is close to $1$.  This behaviour is likely to be seen in other more complex infinite horizon problems, and also in undiscounted finite horizon problems.  The remainder of this paper studies a new stepsize rule that is optimal for the single-state, single-action MDP used in the above analysis.

\section{An optimal stepsize for approximate value iteration}\label{sec:opt}

In Section \ref{sec:derivation}, we derive a new stepsize rule that is optimal for the approximate value iteration problem given by (\ref{eq:vhat}) and (\ref{eq:vbar}). We then study its convergence properties in Section \ref{sec:convergence}. However, while we use the special case in (\ref{eq:vhat})-(\ref{eq:vbar}) for theoretical tractability, our ultimate goal is to obtain an algorithm that can be applied in a general dynamic program. This extension is explained in Section \ref{sec:unknownc}, and the general form of our stepsize is given in (\ref{eq:optadp}). Finally, Section \ref{sec:finitehorizon} considers an extension to finite-horizon problems.

\subsection{Derivation}\label{sec:derivation}

The approximate value iteration problem is given by (\ref{eq:vhat}) and (\ref{eq:vbar}). As before, let $c = \Exp \chat^n$ and $\sigma^2 = Var\left(\chat^n\right)$. Observe that $\vbar^n$ can be written recursively as
\begin{equation}\label{eq:vbarrec}
\vbar^n = \left(1-\alpha_{n-1}\right)\vbar^{n-1} + \alpha_{n-1}\chat^n + \alpha_{n-1}\gamma\vbar^{n-1} = \left(1 - \left(1-\gamma\right)\alpha_{n-1}\right)\vbar^{n-1} + \alpha_{n-1}\chat^n\mbox{.}
\end{equation}
The particular structure of this problem allows us to derive recursive formulas for the mean and variance of the approximation $\vbar^n$. We assume that $\vbar^0 = 0$.

\begin{proposition}\label{prop:deltalambda}
Define
\begin{eqnarray*}
\delta^n = \left\{
\begin{array}{l l}
  \alpha_0 & n = 1\\
  \alpha_{n-1} + \left(1-\left(1-\gamma\right)\alpha_{n-1}\right)\delta^{n-1} & n > 1,
\end{array}
\right.
\end{eqnarray*}
and
\begin{eqnarray*}
\lambda^n = \left\{
\begin{array}{l l}
  \alpha^2_0 & n = 1\\
  \alpha^2_{n-1} + \left(1-\left(1-\gamma\right)\alpha_{n-1}\right)^2\lambda^{n-1} & n > 1.
\end{array}
\right.
\end{eqnarray*}
Then, $\Exp\left(\vbar^n\right) = \delta^n c$ and $Var\left(\vbar^n\right) = \lambda^n \sigma^2$.
\end{proposition}

\begin{proof}
Observe that $\Exp\left(\vbar^1\right) = \alpha_0 c = \delta^1 c$ and $Var\left(\vbar^1\right) = \alpha^2_0 \sigma^2 = \lambda^1 \sigma^2$. Now suppose that $\Exp\left(\vbar^{n-1}\right) = \delta^{n-1} c$ and $Var\left(\vbar^{n-1}\right) = \lambda^{n-1} \sigma^2$. By (\ref{eq:vbarrec}), we have
\begin{equation*}
\Exp\left(\vbar^n\right) = \left(1 - \left(1-\gamma\right)\alpha_{n-1}\right) \delta^{n-1} c + \alpha_{n-1} c = \delta^n c\mbox{.}
\end{equation*}
Furthermore, $\vbar^{n-1}$ depends only on $\chat^{n'}$ for $n' < n$, therefore $\vbar^{n-1}$ and $\chat^n$ are independent. Consequently,
\begin{equation*}
Var\left(\vbar^n\right) = \left(1 - \left(1-\gamma\right)\alpha_{n-1}\right)^2 \lambda^{n-1} \sigma^2 + \alpha^2_{n-1} \sigma^2 = \lambda^n \sigma^2
\end{equation*}
as required.
\end{proof}


The next result shows that these quantities are uniformly bounded in $n$; the proof is given in the Appendix.

\begin{proposition}\label{prop:bounds}
For all $n$, $\delta^n \leq \frac{1}{1-\gamma}$ and $\lambda^n \leq \frac{1}{\gamma\left(1-\gamma\right)}$.
\end{proposition}

We define the optimal stepsize for time $n$ to be the value that achieves
\begin{equation}\label{eq:obj}
\min_{\alpha_{n-1} \in \left[0,1\right]} \Exp\left[\left(\vbar^n\left(\alpha_{n-1}\right) - \Exp \vhat^n\right)^2\right]\mbox{,}
\end{equation}
which is the minimum squared deviation of the time-$n$ estimate $\vbar^n$ from the mean of the time-$n$ observation $\vhat^n$. The constraint $\alpha_{n-1} \in\left[0,1\right]$ is standard in ADP, but turns out to be redundant here; as we see below (Corollary \ref{cor:zeroone}), minimizing the unconstrained objective will produce a solution that always satisfies the constraint.

We can simplify the objective function in (\ref{eq:obj}) in the following manner:
\begin{eqnarray*}
\Exp\left[\left(\vbar^n\left(\alpha_{n-1}\right) - \Exp \vhat^n\right)^2\right] &=& \Exp\left[\left(\left(1-\alpha_{n-1}\right)\vbar^{n-1} + \alpha_{n-1}\vhat^n - \Exp \vhat^n\right)^2\right]\\
&=& \Exp\left[\left(\left(1-\alpha_{n-1}\right)\left(\vbar^{n-1} - \Exp \vhat^n\right) + \alpha_{n-1}\left(\vhat^n - \Exp \vhat^n\right)\right)^2\right]\\
&=& \left(1-\alpha_{n-1}\right)^2 \Exp\left[\left(\vbar^{n-1}-\Exp \vhat^n\right)^2\right] + \alpha^2_{n-1}\Exp\left[\left(\vhat^n - \Exp \vhat^n\right)^2\right]\\
&\,& + 2\alpha_{n-1}\left(1-\alpha_{n-1}\right)\Exp\left[\left(\vbar^{n-1}-\Exp \vhat^n\right)\left(\vhat^n - \Exp \vhat^n\right)\right]\mbox{.}
\end{eqnarray*}
The first equality is obtained using the recursive formula for $\vbar^n$ from (\ref{eq:vbarrec}). Observe that
\begin{eqnarray*}
\Exp\left[\left(\vbar^{n-1}-\Exp \vhat^n\right)\left(\vhat^n - \Exp \vhat^n\right)\right] &=& \Exp\left(\vbar^{n-1}\vhat^n\right) - \Exp\vbar^{n-1}\Exp\vhat^n\\
&=& Cov\left(\vbar^{n-1},\vhat^n\right)\mbox{,}
\end{eqnarray*}
whence we obtain
\begin{eqnarray}\label{eq:firstcov}
\Exp\left[\left(\vbar^n\left(\alpha_{n-1}\right) - \Exp \vhat^n\right)^2\right] &=& \left(1-\alpha_{n-1}\right)^2 \Exp\left[\left(\vbar^{n-1}-\Exp \vhat^n\right)^2\right] + \alpha^2_{n-1}\Exp\left[\left(\vhat^n - \Exp \vhat^n\right)^2\right]\nonumber\\
&\,& + 2\alpha_{n-1}\left(1-\alpha_{n-1}\right)Cov\left(\vbar^{n-1},\vhat^n\right)\mbox{.}
\end{eqnarray}
The error-minimizing stepsize is unique, due to the convexity of the prediction error; the proof of this property is given in the Appendix.

\begin{proposition}\label{prop:convex}
The objective function in (\ref{eq:obj}) is convex in $\alpha_{n-1}$.
\end{proposition}

Due to Proposition \ref{prop:convex}, we can solve (\ref{eq:obj}) by setting the derivative of the prediction error equal to zero and solving for $\alpha_{n-1}$. This yields an equation
\begin{equation*}
\left(\alpha_{n-1}-1\right)\Exp\left[\left(\vbar^{n-1}-\Exp \vhat^n\right)^2\right] + \alpha_{n-1}\Exp\left[\left(\vhat^n - \Exp \vhat^n\right)^2\right] + \left(1-2\alpha_{n-1}\right)Cov\left(\vbar^{n-1},\vhat^n\right) = 0
\end{equation*}
whence we obtain
\begin{equation}\label{eq:firstopt}
\alpha_{n-1} = \frac{\Exp\left[\left(\vbar^{n-1}-\Exp \vhat^n\right)^2\right]-Cov\left(\vbar^{n-1},\vhat^n\right)}{\Exp\left[\left(\vbar^{n-1}-\Exp \vhat^n\right)^2\right] + \Exp\left[\left(\vhat^n - \Exp \vhat^n\right)^2\right]-2Cov\left(\vbar^{n-1},\vhat^n\right)}\mbox{.}
\end{equation}
We now present our main result, which gives an explicit formula for (\ref{eq:firstopt}).

\begin{theorem}
Assuming that $\alpha_0$ is given, the optimal time-$n$ stepsize can be computed using the formula
\begin{equation}\label{eq:opt}
\alpha_{n-1} = \frac{\left(1-\gamma\right)\lambda^{n-1} \sigma^2 + \left(1 - \left(1 - \gamma\right)\delta^{n-1}\right)^2 c^2}{\left(1-\gamma\right)^2\lambda^{n-1} \sigma^2 + \left(1 - \left(1 - \gamma\right)\delta^{n-1}\right)^2 c^2 + \sigma^2} \qquad n = 2,3,...
\end{equation}
where $\delta^{n-1}$ and $\lambda^{n-1}$ are as in Proposition \ref{prop:deltalambda}.
\end{theorem}

\begin{proof}
We compute each expectation in (\ref{eq:firstopt}). First, observe that
\begin{equation*}
\Exp\left[\left(\vhat^n - \Exp \vhat^n\right)^2\right] = Var\left(\vhat^n\right) = Var\left(\chat^n + \gamma \vbar^{n-1}\right) = \left(1 + \gamma^2 \lambda^{n-1}\right)\sigma^2
\end{equation*}
using the independence of $\chat^n$ and $\vbar^{n-1}$ together with Proposition \ref{prop:deltalambda}. We now use a bias-variance decomposition \citep[see e.g.][]{HaTiFr01} to write
\begin{eqnarray}
\Exp\left[\left(\vbar^{n-1}-\Exp \vhat^n\right)^2\right] &=& \Exp\left[\left(\vbar^{n-1} -\Exp \vbar^{n-1} + \Exp\vbar^{n-1}  -\Exp \vhat^n\right)^2\right]\nonumber\\
&=& \Exp\left[\left(\vbar^{n-1} -\Exp \vbar^{n-1}\right)^2\right] + \left(\Exp\vbar^{n-1}  -\Exp \vhat^n\right)^2\nonumber\\
&=& Var\left(\vbar^{n-1}\right) + \left(\Exp\vbar^{n-1}  -\Exp \vhat^n\right)^2\label{eq:biasvar}
\end{eqnarray}
where the cross term vanishes because the quantity $\Exp\vbar^{n-1} - \Exp \vhat^n$ is deterministic, and thus
\begin{equation*}
\Exp\left[\left(\vbar^{n-1} -\Exp \vbar^{n-1}\right)\left(\Exp\vbar^{n-1}  -\Exp \vhat^n\right)\right] = \left(\Exp\vbar^{n-1} - \Exp \vhat^n\right)\Exp\left(\vbar^{n-1} - \Exp \vbar^{n-1}\right) = 0\mbox{.}
\end{equation*}
By Proposition \ref{prop:deltalambda} $Var\left(\vbar^{n-1}\right) = \lambda^{n-1} \sigma^2$, and
\begin{equation*}
\Exp \vhat^n - \Exp \vbar^{n-1} = c + \gamma \delta^{n-1} c - \delta^{n-1} c = \left(1 - \left(1-\gamma\right)\delta^{n-1}\right) c
\end{equation*}
represents the bias of $\vbar^{n-1}$ in predicting $\vhat^n$. Thus,
\begin{equation*}
\Exp\left[\left(\vbar^{n-1}-\Exp \vhat^n\right)^2\right] = \lambda^{n-1} \sigma^2 + \left(1 - \left(1 - \gamma\right)\delta^{n-1}\right)^2 c^2\mbox{.}
\end{equation*}
Finally, we compute
\begin{eqnarray}
Cov\left(\vbar^{n-1},\vhat^n\right) &=& \Exp\left(\vbar^{n-1}\vhat^n\right) - \Exp \vbar^{n-1}\Exp\vhat^n\nonumber\\
&=& \Exp\left( \vbar^{n-1}\left(\chat^n + \gamma\vbar^{n-1}\right)\right) - \Exp\vbar^{n-1}\Exp\left(\chat^n + \gamma\vbar^{n-1}\right)\nonumber\\
&=& c \Exp\vbar^{n-1} + \gamma\Exp\left(\vbar^{n-1}\right)^2 - c \Exp\vbar^{n-1} - \gamma\left(\Exp\vbar^{n-1}\right)^2\nonumber\\
&=& \gamma Var\left(\vbar^{n-1}\right)\mbox{,}\label{eq:cov}
\end{eqnarray}
where we use the independence of $\vbar^{n-1}$ and $\chat^n$ to obtain the third line. Substituting all of these expressions into (\ref{eq:firstopt}) completes the proof.
\end{proof}

\begin{corollary}\label{cor:zeroone}
For all $n$, $\alpha_{n-1} \in \left[0,1\right]$.
\end{corollary}

\begin{proof}
The positivity of $\alpha_{n-1}$ is obvious from (\ref{eq:opt}), where both the numerator and denominator are sums of positive terms (it can easily be seen that $\lambda^{n-1} \geq 0$ for all $n$). To show that $\alpha_{n-1} \leq 1$, first observe that
\begin{equation*}
\gamma\left(1-\gamma\right)\lambda^{n-1} \sigma^2 \leq \sigma^2
\end{equation*}
by the result of Proposition \ref{prop:bounds}. From this it can easily be shown that
\begin{equation*}
\left(1-\gamma\right)\lambda^{n-1}\sigma^2 \leq \left(1-\gamma\right)^2 \lambda^{n-1}\sigma^2 + \sigma^2\mbox{,}
\end{equation*}
completing the proof.
\end{proof}

We see that both the numerator and the denominator of the fraction in (\ref{eq:opt}) include covariance terms. To our knowledge, this is the first stepsize in the literature to explicitly account for the dependence between observations. Furthermore, the formula includes a closed-form expression for the bias $\Exp \vbar^{n-1} - \Exp \vhat^n$, which is balanced against the variance of $\vbar^{n-1}$.

We close this section by showing that our formula behaves correctly in special cases. If the rewards we collect are deterministic, then our estimate $\vbar^n$ is simply adding up the discounted rewards, and should converge to $v^*$ under the optimal stepsize rule. If the process $\vhat^n$ is stationary, i.e. $\gamma = 0$, then $\vbar^n$ is simply estimating $c$, and we should be using the known optimal stepsize rule of $\alpha_{n-1} = \frac{1}{n}$.

\begin{corollary}
If the underlying reward process has zero noise, then $\sigma^2 = 0$ and $\alpha_{n-1} = 1$ for all $n$. It follows that $\vbar^n = \vhat^n$ for all $n$, and
\begin{equation*}
\lim_{n\rightarrow \infty} \vbar^n = \sum^{\infty}_{i=0} \gamma^i c = \frac{c}{1-\gamma}\mbox{.}
\end{equation*}
\end{corollary}

\begin{corollary}\label{corol:stat}
If the problem is stationary, that is, $\gamma = 0$, then the optimal stepsize is given by $\alpha_{n-1} = \frac{1}{n}$ for all $n$ as long as $\alpha_0 = 1$.
\end{corollary}

\begin{proof}
If $\alpha_1 = 1$ and $\gamma = 0$, then $\vhat^n = \chat^n$. It can easily be shown by induction that $\Exp\vbar^n = c$ for all $n$, which means that $\delta^n = 1$ for all $n$. Then, (\ref{eq:opt}) reduces to
\begin{equation*}
\alpha_{n-1} = \frac{\lambda^{n-1}\sigma^2}{\left(1 + \lambda^{n-1}\right)\sigma^2 } = \frac{\lambda^{n-1}}{1 + \lambda^{n-1}}\mbox{.}
\end{equation*}
We claim that $\lambda^{n-1} = \frac{1}{n-1}$. It is clearly true that $\lambda^1 = 1$, from which it follows that $\alpha_1 = \frac{1}{2}$. Now suppose that $\alpha_{n-2} = \frac{1}{n-1}$ and $\lambda^{n-2} = \frac{1}{n-2}$. Then,
\begin{equation*}
\lambda^{n-1} = \alpha^2_{n-2} + \left(1-\alpha_{n-2}\right)^2\lambda^{n-2} = \frac{1}{\left(n-1\right)^2} + \frac{n-2}{\left(n-1\right)^2} = \frac{n-1}{\left(n-1\right)^2} = \frac{1}{n-1}
\end{equation*}
and $\alpha_{n-1} = \frac{1}{n}$, as required.
\end{proof}

\subsection{Convergence analysis}\label{sec:convergence}

It is well-known \citep{KuYi97,Ts94} that, with some regularity assumptions on the underlying stochastic processes, a stochastic approximation algorithm is provably convergent as long as $\alpha_{n-1} \geq 0$ for all $n$ and
\begin{equation*}
\sum^{\infty}_{n=1}\alpha_{n-1} = \infty\mbox{,} \qquad \sum^{\infty}_{n=1}\alpha^2_{n-1} < \infty\mbox{.}
\end{equation*}
We show the first condition by establishing a lower bound on $\alpha_{n-1}$. The proof is given in the Appendix.

\begin{proposition}\label{prop:lowbound}
For all $n \geq 1$, $\alpha_{n-1} \geq \frac{1-\gamma}{n}$.
\end{proposition}

From Proposition \ref{prop:lowbound}, it follows that
\begin{equation*}
\sum^{\infty}_{n=1}\alpha_{n-1} \;\geq \;\left(1-\gamma\right)\,\sum^{\infty}_{n=1}\frac{1}{n} \;=\; \infty\mbox{,}
\end{equation*}
satisfying one of the conditions for convergence. The second condition $\sum^{\infty}_{n=1}\alpha^2_{n-1}<\infty$ can sometimes be relaxed to the requirement that $\alpha_{n-1}\rightarrow 0$. For instance, \cite{KuYi97} discusses the sufficiency of this requirement in stochastic approximation problems with bounded observations. See also \cite{BrCiZe11} for recent proofs of convergence with weaker conditions on the stepsizes. We do not show almost sure convergence in this paper, but we do show that $\alpha_{n-1}\rightarrow 0$, a condition that is common to the above convergence proofs. While this does not automatically imply a.s. convergence, it does produce convergence in $L^2$ for the single-state, single-action model.

We begin by showing that the bias term in the stepsize formula converges to zero; the proof is given in the Appendix. We then prove that $\alpha_{n-1} \rightarrow 0$.

\begin{proposition}\label{prop:biasgoestozero}
$\lim_{n\rightarrow\infty} \delta^n = \frac{1}{1-\gamma}$.
\end{proposition}

\begin{theorem}\label{thm:limitzero}
$\lim_{n\rightarrow\infty} \alpha_{n-1} = 0$.
\end{theorem}

\begin{proof}
It is enough to show that $\lambda^n \rightarrow 0$ and apply (\ref{eq:opt}) together with Proposition \ref{prop:biasgoestozero}. We show that every convergent subsequence of $\lambda^n$ must converge to zero using a proof by contradiction.

First, suppose that $n_k$ is a subsequence satisfying $\lim_{k\rightarrow \infty} \lambda^{n_k} = \ell$. Combining this with Proposition \ref{prop:biasgoestozero}, we return to (\ref{eq:opt}) and find
\begin{equation*}
\lim_{k\rightarrow \infty} \alpha_{n_k} = \frac{\left(1-\gamma\right) \ell}{\left(1-\gamma\right)^2 \ell + 1}.
\end{equation*}
We then return to Proposition \ref{prop:deltalambda} and derive
\begin{eqnarray}
\lim_{k\rightarrow\infty} \lambda^{n_k + 1} &=& \left[\frac{\left(1-\gamma\right)\ell}{\left(1-\gamma\right)^2 \ell + 1}\right]^2 + \left[1 -\frac{\left(1-\gamma\right)^2 \ell}{\left(1-\gamma\right)^2 \ell + 1}\right]^2 \ell\nonumber\\
&=& \frac{\ell}{\left(1-\gamma\right)^2 \ell + 1}.\label{eq:smallerlimit}
\end{eqnarray}
It follows from (\ref{eq:smallerlimit}) that, if $\ell > 0$, then
\begin{equation}\label{eq:strictlysmallerlimit}
\lim_{k\rightarrow\infty} \lambda^{n_k+1} < \lim_{k\rightarrow\infty} \lambda^{n_k}.
\end{equation}
By Proposition \ref{prop:bounds}, we know that the sequence $\left(\lambda^n\right)^{\infty}_{n=1}$ is bounded. Therefore, the set of accumulation points for this sequence is closed and bounded. Suppose that
\begin{equation*}
\limsup_{n\rightarrow\infty} \lambda^n = \lambda^*
\end{equation*}
and that $\lambda^* > 0$. Let $n_k$ be a subsequence with $\lambda^{n_k} \rightarrow \lambda^*$. The subsequence $\left(\lambda^{n_k - 1}\right)^{\infty}_{k=1}$ is bounded, and therefore must contain an additional convergent subsequence, which we denote by $m_k$. Suppose that $\lim_{k\rightarrow\infty} \lambda^{m_k} = \ell$. It must be the case that
\begin{equation*}
\lim_{k\rightarrow\infty} \lambda^{m_k+1} = \lim_{k\rightarrow\infty} \lambda^{n_k} = \lambda^*.
\end{equation*}
This implies that $\ell > 0$, because otherwise (\ref{eq:smallerlimit}) would imply that $\lambda^* = 0$. However, it then follows from (\ref{eq:strictlysmallerlimit}) that $\lambda^* < \ell$. This is impossible, because we took $\lambda^*$ to be the largest accumulation point of the sequence $\left(\lambda^n\right)^{\infty}_{n=1}$. It must therefore be the case that
\begin{equation*}
\limsup_{n\rightarrow\infty} \lambda^n = 0,
\end{equation*}
whence $\lambda^n \rightarrow 0$, as required.
\end{proof}

It follows immediately from these results that $\bar{v}^n\rightarrow v^*$ in $L^2$ and in probability. Observe that
\begin{eqnarray}
\Exp\left[\left(\bar{v}^n - v^*\right)^2\right] &=& \Exp\left[\left(\bar{v}^n - \Exp\bar{v}^n + \Exp\bar{v}^n - v^*\right)^2\right]\nonumber\\
&=& Var\left(\bar{v}^n\right) + \left(\Exp \bar{v}^n - v^*\right)^2\nonumber\\
&=& \lambda^n \sigma^2 + \left(\delta^n c - \frac{c}{1-\gamma}\right)^2.\label{eq:l2}
\end{eqnarray}
Together, Proposition \ref{prop:biasgoestozero} and Theorem \ref{thm:limitzero} imply that (\ref{eq:l2}) vanishes to zero as $n\rightarrow\infty$.

Proposition \ref{prop:lowbound} along with Theorem \ref{thm:limitzero} have important practical as well as theoretical implications. The lower bound provided by Proposition \ref{prop:lowbound} ensures that the stepsize will not decline too quickly. While this bound is provided by many standard rules, including $1/n$, we now have the benefit of a rule that is designed to minimize prediction error for faster convergence, but is still guaranteed to avoid the risk of stalling. The guarantee in Theorem \ref{thm:limitzero} that the stepsize will asymptotically approach zero is particularly valuable in applications where we are interested not just in the policy, but in the values themselves. For example, in finance, the value function is used to estimate the price of an option. In the fleet management application of \cite{SiDaGeGiNiPo09}, the value functions were used to estimate the marginal value of truck drivers. In both applications, it is essential to have an algorithm that will produce tight estimates of these values.

\subsection{Algorithmic procedure for general dynamic programs}\label{sec:unknownc}

We now discuss how (\ref{eq:opt}) can be adapted for a general dynamic program. The first step is to consider an extension of the single-state model where $c$ and $\sigma^2$ are unknown. In this case, we estimate the unknown quantities by smoothing on the observations $\chat^n$ and plugging these estimates into the expression for the optimal stepsize \citep[this is known as the plug-in principle; see e.g.][]{BiDo01}. Let
\begin{eqnarray}
\bar{c}^n &=& \left(1-\nu_{n-1}\right)\cbar^{n-1} + \nu_{n-1}\chat^n\label{eq:cbar}\\
\left(\bar{\sigma}^n\right)^2 &=& \left(1-\nu_{n-1}\right)\left(\bar{\sigma}^{n-1}\right)^2 + \nu_{n-1}\left(\chat^n - \bar{c}^{n-1}\right)^2\label{eq:sigma}
\end{eqnarray}
represent our estimates of the mean and variance of the rewards. The secondary stepsize $\nu_{n-1}$ is chosen according to some deterministic stepsize rule (e.g. set to a constant). Then, (\ref{eq:opt}) becomes
\begin{equation}\label{eq:optapp}
\alpha_{n-1} = \frac{\left(1-\gamma\right)\lambda^{n-1} \left(\bar{\sigma}^{n}\right)^2 + \left(1 - \left(1 - \gamma\right)\delta^{n-1}\right)^2 \left(\cbar^{n}\right)^2}{\left(1-\gamma\right)^2\lambda^{n-1} \left(\bar{\sigma}^{n}\right)^2 + \left(1 - \left(1 - \gamma\right)\delta^{n-1}\right)^2 \left(\cbar^{n}\right)^2 + \left(\bar{\sigma}^{n}\right)^2} \qquad n = 2,3,...
\end{equation}
where $\delta^{n-1}$, $\lambda^{n-1}$ are computed the same way as before.

At first glance, it appears that we run into the problem of needing a secondary stepsize to calculate an optimal one. However, it is important to note that the secondary stepsize $\nu_{n-1}$ is only required to estimate the parameters of the distribution of the one-period reward $\chat^n$. Unlike the sequence $\vbar^n$ of value function approximations, the one-period reward in the single-state problem is \textit{stationary}, and can be straightforwardly estimated from the random rewards collected in each time period.

The true significance of (\ref{eq:optapp}) is that it can be easily extended to a general MDP with many states and actions. In this case, we replace the random reward $\chat^n$ in (\ref{eq:cbar})-(\ref{eq:sigma}) by the one-period reward $C(S^n,x^n)$ earned by taking action $x^n$ in state $S^n$. The sequence of these rewards depends on the policy used to visit states; in approximate value iteration, this policy will change over time, thus making the sequence of rewards non-stationary. However, as discussed in Section \ref{sec:model}, the basic value-iteration update of (\ref{eq:basicadp}) eventually converges to an optimal policy, meaning that the expected one-period reward earned in a state converges to a single system-wide constant $\cbar$. This suggests that, in a general DP, it is sufficient to keep one single system-wide estimate $\bar{c}^n$ (and similarly $\bar{\sigma}^n$) rather than to store state-dependent estimates.

On the other hand, the quantities $\delta^n$ and $\lambda^n$ are related to the bias and variance of the value function approximation. This suggests that, in a general DP, they should be state-dependent. For example, if we use the Q-learning algorithm, we will have a separate approximation for each state-action pair, leading to a state-dependent stepsize. Figure \ref{fig:alg} describes an example implementation of OSAVI in a classic finite-state, finite-action MDP where a generic ADP algorithm is used with a lookup table approximation. In a more complex problem, if we employ a state aggregation method such as that of \cite{GePoKu08}, we would store a different $\delta^n$ and $\lambda^n$ for each block of the aggregation structure. The memory cost is similar to the procedure in \cite{ScZhLe13}, where two recursively updated quantities are stored for each estimated parameter.

\begin{figure}[h!]
 \mbox{}\hrulefill\mbox{}
 \begin{description}
    \item[1:] Initialize $\bar{V}^0(S,x)$ and $\alpha_0(S,x)$ for all $(S,x)$. Set $\delta^1(S,x) = \alpha_0(S,x)$ and $\lambda^1(S,x) = \alpha_0(S,x)^2$. Also initialize $\bar{c}^0$, $\bar{\sigma}^0$, $S^0$, and $x^0$. 
    \item[2:] Set $n = 1$, and generate $S^1$ from the transition function.
    \item[3:] Solve
    \begin{equation*}
    \hat{v}^n = \max_{x \in \mathcal{X}} C\left(S^n,x\right) + \gamma\bar{V}^{n-1}\left(S^n,x\right)
    \end{equation*}
    and let $x^n$ be the value of $x$ that achieves the maximum.
    \item[4:] Update the system-wide parameters
    \begin{eqnarray*}
    \bar{c}^n &=& \left(1-\nu_{n-1}\right)\bar{c}^{n-1} + \nu_{n-1}C\left(S^n,x^n\right),\\
    \left(\bar{\sigma}^n\right)^2 &=& \left(1-\nu_{n-1}\right)\left(\bar{\sigma}^{n-1}\right)^2 + \nu_{n-1}\left(C\left(S^n,x^n\right) - \bar{c}^{n-1}\right)^2.    
    \end{eqnarray*}
    \item[5:] If $n > 1$, calculate
    \begin{eqnarray}
    \hspace{-0.10in} &\,& \hspace{-0.10in} \alpha_{n-1}\left(S^{n-1},x^{n-1}\right)\\\nonumber
    \hspace{-0.10in} &=& \hspace{-0.10in} \frac{\left(1-\gamma\right)\lambda^{n-1}\left(S^{n-1},x^{n-1}\right) \left(\bar{\sigma}^{n}\right)^2 + \left(1 - \left(1 - \gamma\right)\delta^{n-1}\left(S^{n-1},x^{n-1}\right)\right)^2 \left(\cbar^{n}\right)^2}{\left(1-\gamma\right)^2\lambda^{n}\left(S^{n-1},x^{n-1}\right) \left(\bar{\sigma}^{n}\right)^2 + \left(1 - \left(1 - \gamma\right)\delta^{n-1}\left(S^{n-1},x^{n-1}\right)\right)^2 \left(\cbar^{n}\right)^2 + \left(\bar{\sigma}^{n}\right)^2}.\label{eq:optadp}
    \end{eqnarray}
    \item[6:] Update the value function approximation using
    \begin{equation*}
    \bar{V}^n\left(S^{n-1},x^{n-1}\right) = \left(1-\alpha_{n-1}\right)\bar{V}^{n-1}\left(S^{n-1},x^{n-1}\right) + \alpha_{n-1}\hat{v}^n.
    \end{equation*}
    \item[7:] Update the stepsize parameters using
    \begin{eqnarray*}
      \delta^n\left(S^{n-1},x^{n-1}\right) \hspace{-0.10in} &=& \hspace{-0.10in} \alpha_{n-1}\left(S^{n-1},x^{n-1}\right) + \left(1-\left(1-\gamma\right)\alpha_{n-1}\left(S^{n-1},x^{n-1}\right)\right)\delta^{n-1}\left(S^{n-1},x^{n-1}\right),\\
      \lambda^n\left(S^{n-1},x^{n-1}\right) \hspace{-0.10in} &=& \hspace{-0.10in} \alpha^2_{n-1}\left(S^{n-1},x^{n-1}\right) + \left(1-\left(1-\gamma\right)\alpha_{n-1}\left(S^{n-1},x^{n-1}\right)\right)^2\lambda^{n-1}\left(S^{n-1},x^{n-1}\right).
    \end{eqnarray*}
    \item[8:] Generate $S^{n+1}$ from the transition function or by following a target policy.
    \item[9:] Increment $n$ and return to Step 3.
 \end{description}
 \mbox{}\hrulefill\mbox{}
 \caption{Example implementation of infinite-horizon OSAVI in a finite-state, finite-action MDP with a generic ADP algorithm.}\label{fig:alg}
\end{figure}

In a general ADP setting, we suggest using a constant stepsize in (\ref{eq:cbar}), e.g. $\nu_{n-1}=0.2$, to avoid giving equal weight to early observations taken in the transient period before the MDP has reached steady state, while the probability of being in a state is still changing with the policy. Our numerical work suggests that performance is not very sensitive to the choice of $\nu_{n-1}$.


Finally, we briefly note that our convergence analysis in Section \ref{sec:convergence} mostly carries over to the general case. First, the bound in Proposition \ref{prop:lowbound} still holds almost surely, since the proof holds for arbitrary values of $c$ and $\sigma$, even if they change between iterations. The bounds in Proposition \ref{prop:bounds} also hold, since the proofs only use the functional forms of the recursive updates for $\delta^n$ and $\lambda^n$, and hold for any arbitrary stepsize sequence. Consequently, Theorem \ref{thm:limitzero} still holds a.s. as long as the sample-based approximations $\bar{c}^n,\bar{\sigma}^n$ do not explode to infinity on any subsequence. If these approximations have a type of convergence (e.g. in probability), we will have $\alpha_{n-1}\rightarrow 0$ also in that sense.

\subsection{Extension to finite horizon}\label{sec:finitehorizon}

While it is possible to solve finite horizons using the same algorithmic strategy, we observe that optimal stepsizes vary systematically as a function of the number of time periods to the end of horizon.  The best stepsize for states at the end of the horizon is very close to $1/n$, because we do not face the need to sum rewards over a horizon.  Optimal stepsizes then increase as we move closer to the first time period.

We can capture this behavior using a finite horizon version (with $T$ time stages) of our single-state, single-action problem. In this setting, approximate value iteration is replaced with approximate dynamic programming. Equations (\ref{eq:vhat}) and (\ref{eq:vbar}) become
\begin{eqnarray}
\vhat^n_t &=& \chat^n_t + \gamma\vbar^{n-1}_{t+1}\\
\vbar^n_t &=& \left(1-\alpha_{n-1,t}\right)\vbar^{n-1}_t + \alpha_{n-1,t} \vhat^n_t\mbox{.}
\end{eqnarray}
These equations are solved for $t = 1,...,T-1$ in each time step $n$. We assume that $\vbar^n_T = 0$ for all $n$, and that the observations $\chat^n_t$ are independent and identically distributed for all $n$ and $t$.

Our analysis can easily be extended to this setting. First, we can obtain expressions for the expected value and variance of $\vbar^n_t$ that generalize our derivations of $\delta^n$ and $\lambda^n$ in Section \ref{sec:derivation}. The following proposition describes these expressions.

\begin{proposition}\label{prop:findl}
For $t = 1,...,T-1$, define
\begin{eqnarray*}
\delta^n_t = \left\{
\begin{array}{l l}
  \alpha_{0,t} & n = 1\\
  \left(1+\gamma\delta^{n-1}_{t+1}\right)\alpha_{n-1,t} + \left(1-\alpha_{n-1,t}\right)\delta^{n-1}_t & n > 1
\end{array}
\right.
\end{eqnarray*}
Also, for $t,t' = 1,...,T-1$, let
\begin{eqnarray*}
\lambda^n_{t,t'} = \left\{
\begin{array}{l l}
  \alpha^2_{0,t}1_{\left\{t = t'\right\}} & n = 1\\
  \alpha^2_{n-1,t}1_{\left\{t=t'\right\}} + J^{n-1}_{t,t'} + K^{n-1}_{t,t'} + L^{n-1}_{t,t'} + M^{n-1}_{t,t'} & n > 1
\end{array}
\right.
\end{eqnarray*}
where
\begin{eqnarray*}
J^{n-1}_{t,t'} &=& \left(1-\alpha_{n-1,t}\right)\left(1-\alpha_{n-1,t'}\right)\lambda^{n-1}_{t,t'}\mbox{,}\\
K^{n-1}_{t,t'} &=& \gamma\left(1-\alpha_{n-1,t}\right)\alpha_{n-1,t'}\lambda^{n-1}_{t,t'+1}\mbox{,}\\
L^{n-1}_{t,t'} &=& \gamma\alpha_{n-1,t}\left(1-\alpha_{n-1,t'}\right)\lambda^{n-1}_{t+1,t'}\mbox{,}\\
M^{n-1}_{t,t'} &=& \gamma^2\alpha_{n-1,t}\alpha_{n-1,t'}\lambda^{n-1}_{t+1,t'+1}\mbox{.}
\end{eqnarray*}
Then, $\Exp\left(\vbar^n_t\right) = \delta^n_t c$ and $Cov\left(\vbar^n_t,\vbar^n_{t'}\right) = \lambda^n_{t,t'} \sigma^2$.
\end{proposition}

The proof uses the same logic as the proof of Proposition \ref{prop:deltalambda}. We can think of $\lambda^n$ as a symmetric matrix that can be updated recursively using the elements of $\lambda^{n-1}$. The matrix starts out diagonal, and as $n$ increases, the covariances gradually expand from the main diagonal outward. Next, we can repeat the analysis of Section \ref{sec:derivation} to solve
\begin{equation*}
\min_{\alpha_{n-1,t}\in\left[0,1\right]} \Exp\left[ \left(\vbar^n_t\left(\alpha_{n-1,t}\right) - \Exp \vhat^n_t\right)^2\right]\mbox{.}
\end{equation*}
The next result gives the solution.

\begin{theorem}\label{thm:finopt}
If $\alpha_{0,t}$ is given, the optimal stepsize for time $t$ at iteration $n$ is given by
\begin{equation}\label{eq:optfin}
\alpha_{n-1,t} = \frac{\left(\lambda^{n-1}_{t,t} - \gamma\lambda^{n-1}_{t,t+1}\right)\sigma^2 + \left(1- \delta^{n-1}_t+\gamma\delta^{n-1}_{t+1}\right)^2 c^2}{\left(\lambda^{n-1}_{t,t}-2\gamma\lambda^{n-1}_{t,t+1} + \gamma^2 \lambda^{n-1}_{t+1,t+1}\right)\sigma^2 + \left(1-\delta^{n-1}_t +\gamma\delta^{n-1}_{t+1}\right)^2 c^2 + \sigma^2}\mbox{.}
\end{equation}
\end{theorem}

In the infinite-horizon case, this reduces to our original formula in (\ref{eq:opt}). The finite-horizon formula requires us to store more parameters in the form of a matrix $\lambda^n$, which has the potential to incur substantially greater computational cost. The benefit is that we can now optimally vary the stepsize by $t$. If $c$ and $\sigma^2$ are unknown, we can adapt the approximation procedure outlined in Section \ref{sec:unknownc}, and replace the unknown values in (\ref{eq:optfin}) with $\cbar^n$ and $\left(\bar{\sigma}^n\right)^2$.

\section{Experimental study: one state, one action}\label{sec:mainexp}

We first study the performance of our stepsize rule on an instance of the single-state, single-action problem. This allows us to obtain insights into the sensitivity of performance with respect to different problem parameters. We considered normally distributed rewards with mean $c = 1$ and standard deviation $\sigma = 1$, with $\gamma = 0.9$ as the discount factor. The optimal value for this problem is $V^* = 10$. All policies used $\bar{v}^0 = 0$ as the initial approximation. Furthermore, all sample-based parameters for these policies (e.g. $\bar{c}^0$ and $\bar{\sigma}^0$ for OSAVI) were initialized to zero here and throughout all parts of our study. Five different stepsize rules were implemented; we briefly describe them as follows.

\textit{Optimal stepsize for approximate value iteration (OSAVI).} We use the approximate version of the optimal stepsize, given by (\ref{eq:optadp}). The secondary stepsize $\nu_{n-1}$ was set to $0.2$.

\textit{Bias-adjusted Kalman filter (OSA/BAKF).} We use the approximate version of the OSA/BAKF algorithm in \cite[Fig. 4]{GePo06}. Like OSAVI, this stepsize minimizes a form of the prediction error for a scalar signal processing problem, but assumes that observations are independent. A secondary stepsize rule $\bar{\nu}_{n-1} = 0.05$ is used to estimate the bias of the value function approximation (unlike OSAVI, which uses a closed-form expression for this quantity).

\textit{McClain's rule.} McClain's stepsize formula is given by
\begin{eqnarray*}
\alpha_{n} = \left\{
\begin{array}{c l}
  1 & \mbox{ if $n=1$}\\
  \frac{\alpha_{n-1}}{1+\alpha_{n-1}-\bar{\alpha}} & \mbox{ otherwise,}
\end{array}
\right.
\end{eqnarray*}
where $\bar{\alpha}$ is a tunable parameter. This stepsize behaves like the $1/n$ rule in early iterations, but quickly converges to the limit point $\bar{\alpha}$, and then behaves more like a constant stepsize rule.  This tends to happen within approximately $10$ iterations. For our experiments, we used $\bar{\alpha} = 0.1$; the issue of tuning $\bar{\alpha}$ is discussed in Section \ref{sec:tunable}. McClain's rule should be viewed as a slightly more sophisticated version of a constant stepsize.

\textit{Harmonic stepsize.} This deterministic rule is given by $\alpha_{n-1} = \frac{a}{a+n}$, where $a > 0$ is a tunable parameter. A value of $a = 10$ yielded good performance for our choice of problem parameters. However, the harmonic stepsize is sensitive to the choice of the tunable parameter $a$, which is highly problem dependent.  If we expect good convergence in a few hundred iterations, $a$ on the order of 5 or 10 may work quite well.  On the other hand, if we anticipate running our algorithm millions of iterations (which is not uncommon in reinforcement learning), we might choose $a$ on the order of 10,000 or higher. This issue is discussed further in Section \ref{sec:tunable}.

\textit{Incremental Delta-Bar-Delta (IDBD)}. This rule, introduced by \cite{Sutton92}, is given by $\alpha_{n-1} = \min\left\{1,\exp\left(\Delta_{n-1}\right)\right\}$, where $\Delta_n = \Delta_{n-1} + \theta\left(\hat{v}^n - \bar{v}^{n-1}\right)h_{n-1}$ and $h_n = \left(1-\alpha_{n-1}\right)h_{n-1} + \alpha_{n-1}\left(\hat{v}^n - \bar{v}^{n-1}\right)$. This is an example of an exponentiated gradient method, where averaging is performed on the logarithm of the stepsize. We used $\theta = 0.001$ as the tunable parameter.

We also considered the polynomial stepsize $\alpha_{n-1} = 1/n^{\beta}$, but it consistently underperformed the rules listed above, and is omitted from the subsequent analysis. The constant rule $\alpha_{n-1} = \bar{\alpha}$ yielded results very similar to McClain's rule, and is also omitted.

\subsection{Numerical evaluation of stepsize rules}

\begin{figure}[t]
\centering
\includegraphics[scale=0.3]{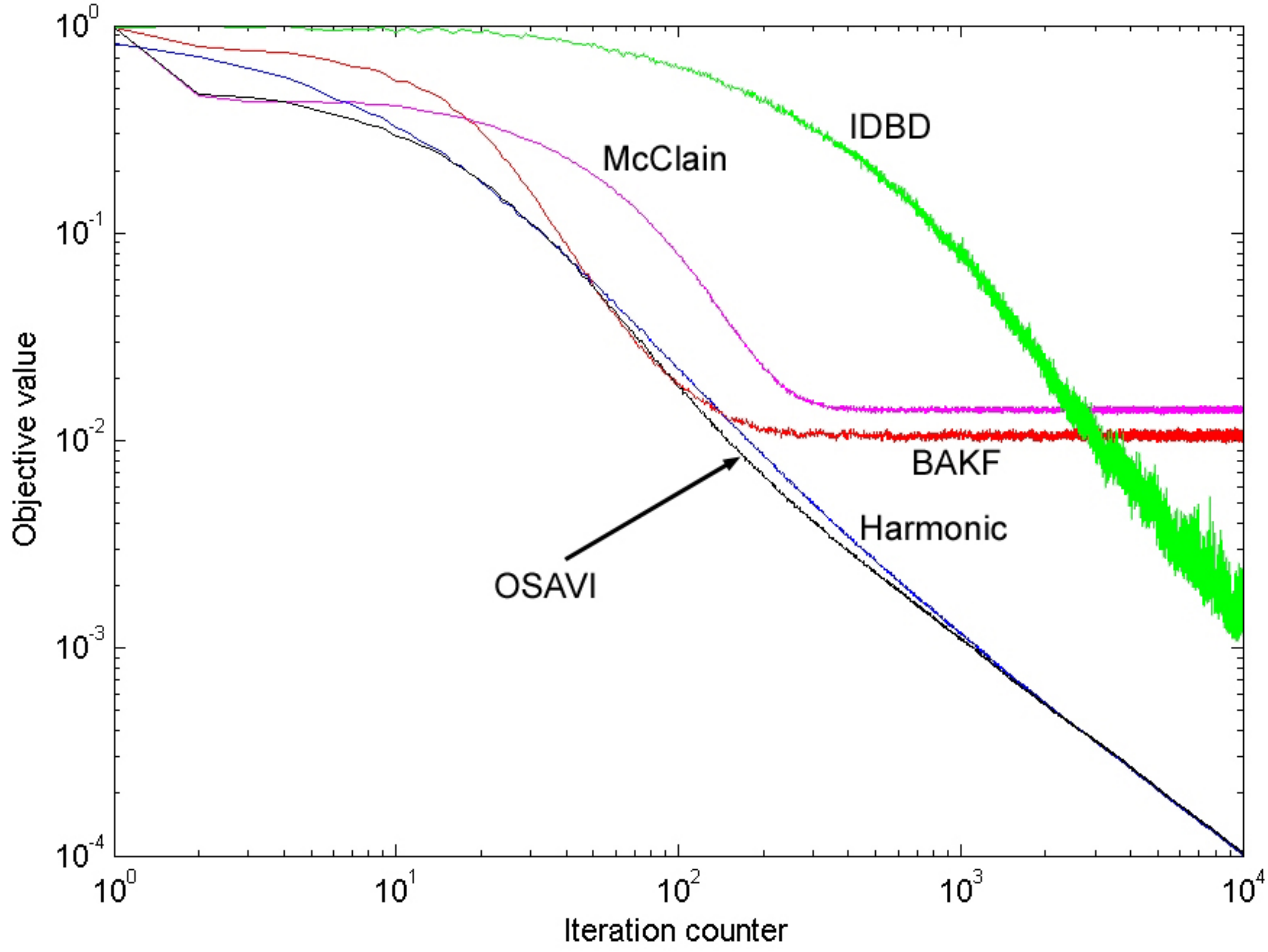}
\caption{Objective values achieved by each stepsize rule over $10^4$ iterations.}
\label{fig:mainplot}
\end{figure}

Figure \ref{fig:mainplot} shows the value of the objective function in (\ref{eq:obj}) achieved by each stepsize rule over $10^4$ iterations. The OSAVI rule consistently achieves the best performance (lowest objective value). However, the harmonic stepsize, when properly tuned, performs comparably. The BAKF and McClain rules level off around an objective value of $10^{-2}$. Each data point in Figure \ref{fig:mainplot} is an average over an ``outer loop'' of $10^4$ simulations.

It should be noted that, while IDBD displays the slowest convergence early on, it eventually overtakes BAKF and McClain's rule and continues to exhibit improvement in the late iterations. In the single-state setting, we found that it was less sensitive to its tunable parameter than the harmonic rule, and also produced less volatile stepsizes than BAKF. We will examine the performance of this rule in multi-stage problems later on. By contrast, the other benchmarks (harmonic, McClain, and BAKF) were fairly sensitive to their tunable parameters. We discuss this below in the context of the single-state problem, which allows us to examine tuning issues with a minimal number of other problem inputs.

\subsection{Discussion of tunable parameters}\label{sec:tunable}

We begin by considering the approximate BAKF rule, which uses a secondary stepsize $\bar{\nu}_{n-1}$ to estimate the bias $\beta^n$. Figure \ref{fig:osatuning} shows the effect of varying $\bar{\nu}_{n-1}$ on the objective value achieved by BAKF (with the optimal stepsize shown for comparison). We see that, when we use a constant value for the secondary stepsize (e.g. $\bar{\nu}_{n-1} = 0.05$), there is a clear tradeoff between performance in the early and late iterations. Smaller values of $\bar{\nu}_{n-1}$ result in better performance in the long run (the objective value achieved by BAKF plateaus at a lower level), but worse performance in the short run. In terms of the quality of our approximation of $V^*$, smaller constants cause slower convergence, but more stable estimates.

\begin{figure}[t]
\centering
\subfigure[BAKF rule.]{
\includegraphics[width=3.1in]{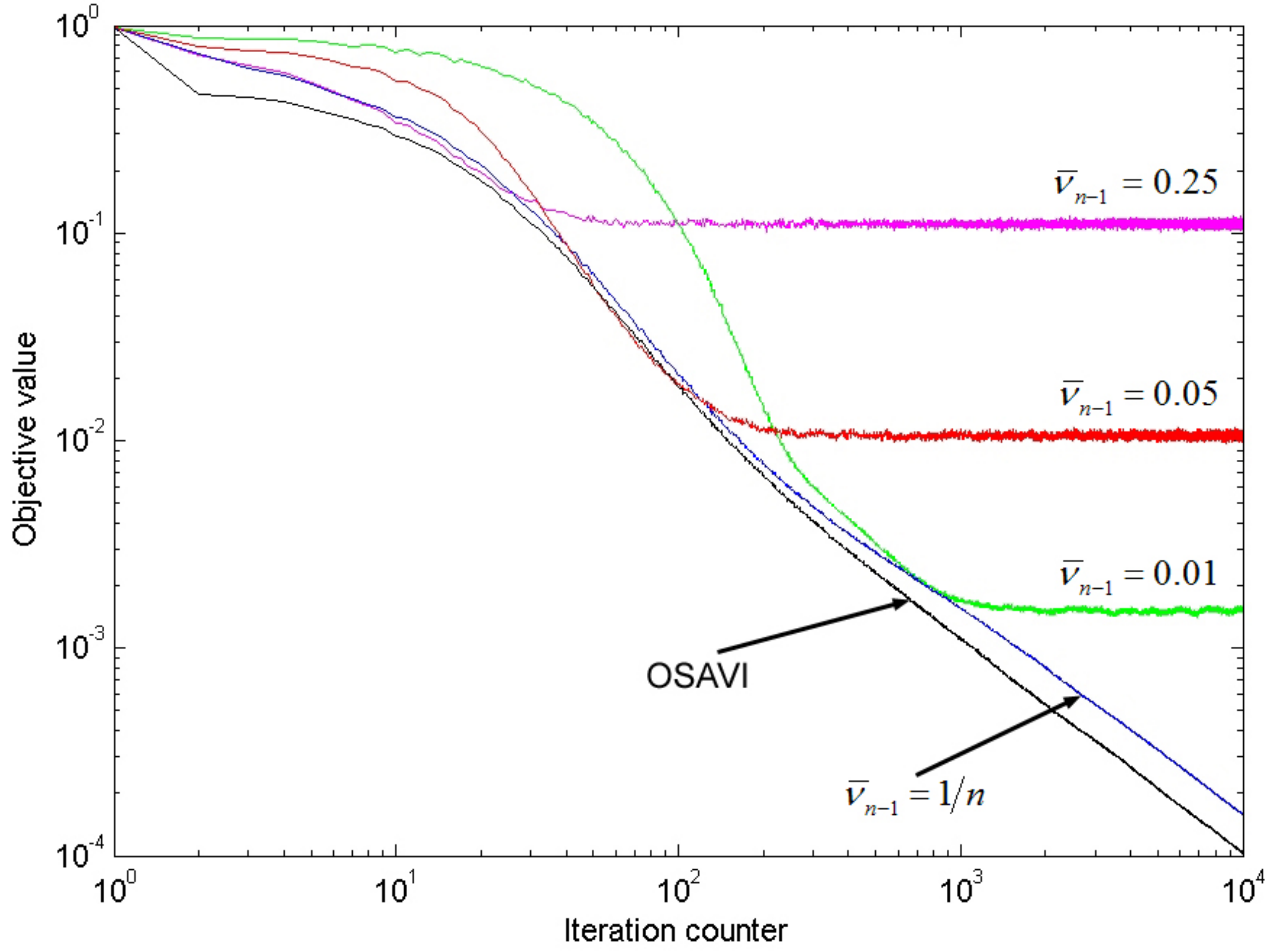}
\label{fig:osatuning}
}
\subfigure[OSAVI rule.]{
\includegraphics[width=3.1in]{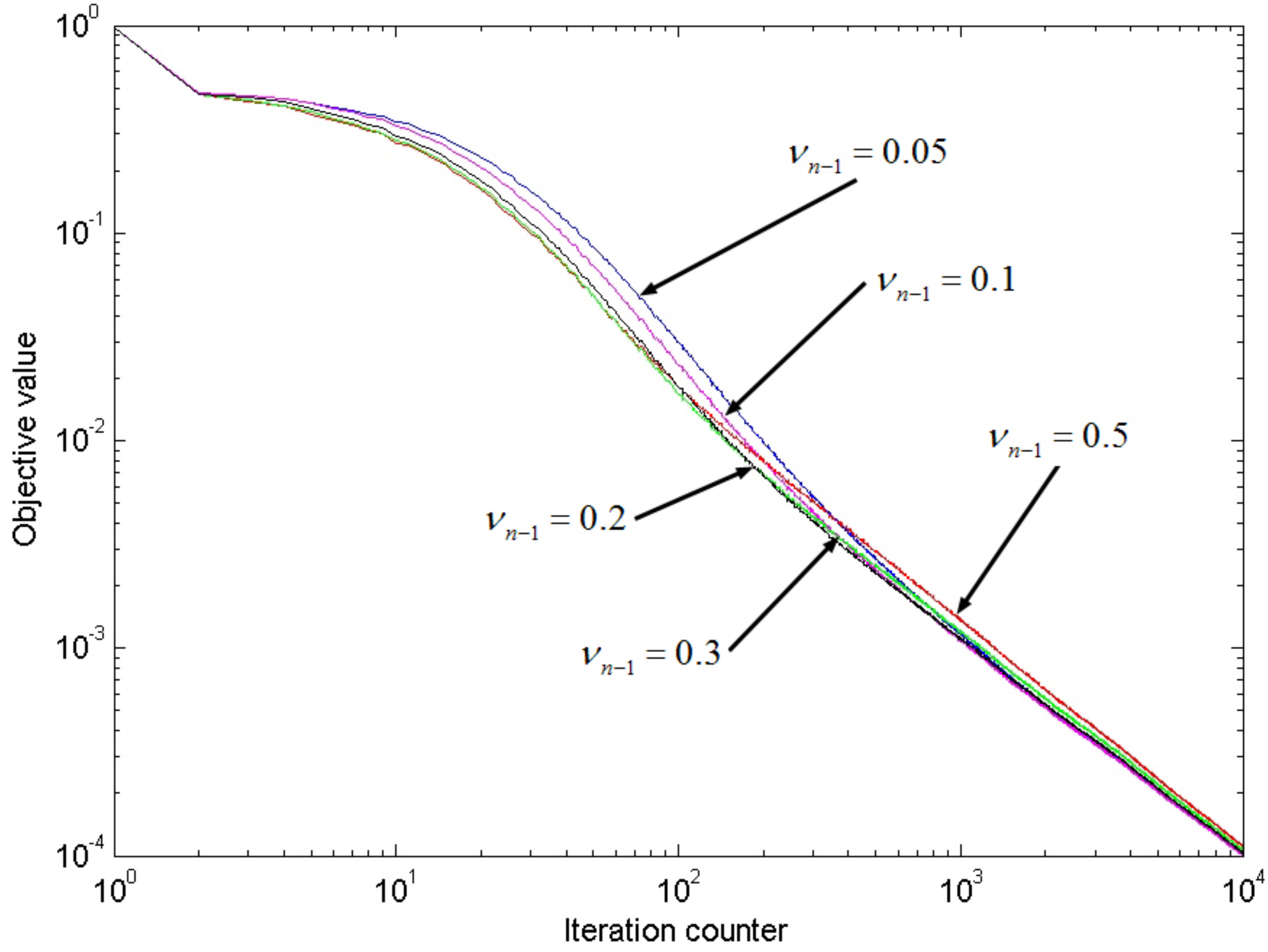}
\label{fig:opttuning}
}
\caption{Effect of the secondary parameter $\nu_{n-1}$ on the objective values achieved by (a) the approximate BAKF rule, and (b) the approximate optimal rule.}
\label{fig:optosatuning}
\end{figure}

It is also necessary to note one special case, where $\bar{\nu}_{n-1} = 1/n$. If we use the $1/n$ rule for the secondary stepsize, the objective value achieved by BAKF declines to zero in the long run. This is not the case when we use a constant stepsize. Furthermore, the use of the $1/n$ rule produces very close performance to that of OSAVI. However, in a general MDP, where there are many different rewards, a constant stepsize may be better able to handle the transient phase of the MDP. For this reason, we focus primarily on constant values of $\bar{\nu}_{n-1}$ in this study.

Even with a declining secondary stepsize, the BAKF rule is outperformed by OSAVI with a simple constant secondary stepsize of $\nu_{n-1} = 0.2$. The results for different values of $\bar{\nu}_{n-1}$ indicate that BAKF is quite sensitive to the choice of $\bar{\nu}_{n-1}$.

Figure \ref{fig:opttuning} suggests that OSAVI is relatively insensitive to the choice of secondary stepsize. The lines in Figure \ref{fig:opttuning} represent the performance of OSAVI for values of $\nu_{n-1}$ ranging from $0.05$ to as high as $0.5$. We see that these changes have a much smaller effect on the performance of OSAVI than varying $\bar{\nu}_{n-1}$ had on the BAKF rule. Very small values of $\nu_{n-1}$, such as $0.05$, do yield slightly poorer performance, but there is little difference between $0.2$ and $0.5$. Furthermore, the objective value achieved by OSAVI declines to zero for each constant value of $\nu_{n-1}$, whereas BAKF always levels off under a constant secondary stepsize. We conclude that OSAVI is more robust than BAKF, and requires less tuning of the secondary stepsize.

\begin{figure}[t]
\centering
\subfigure[McClain's rule.]{
\includegraphics[width=3.1in]{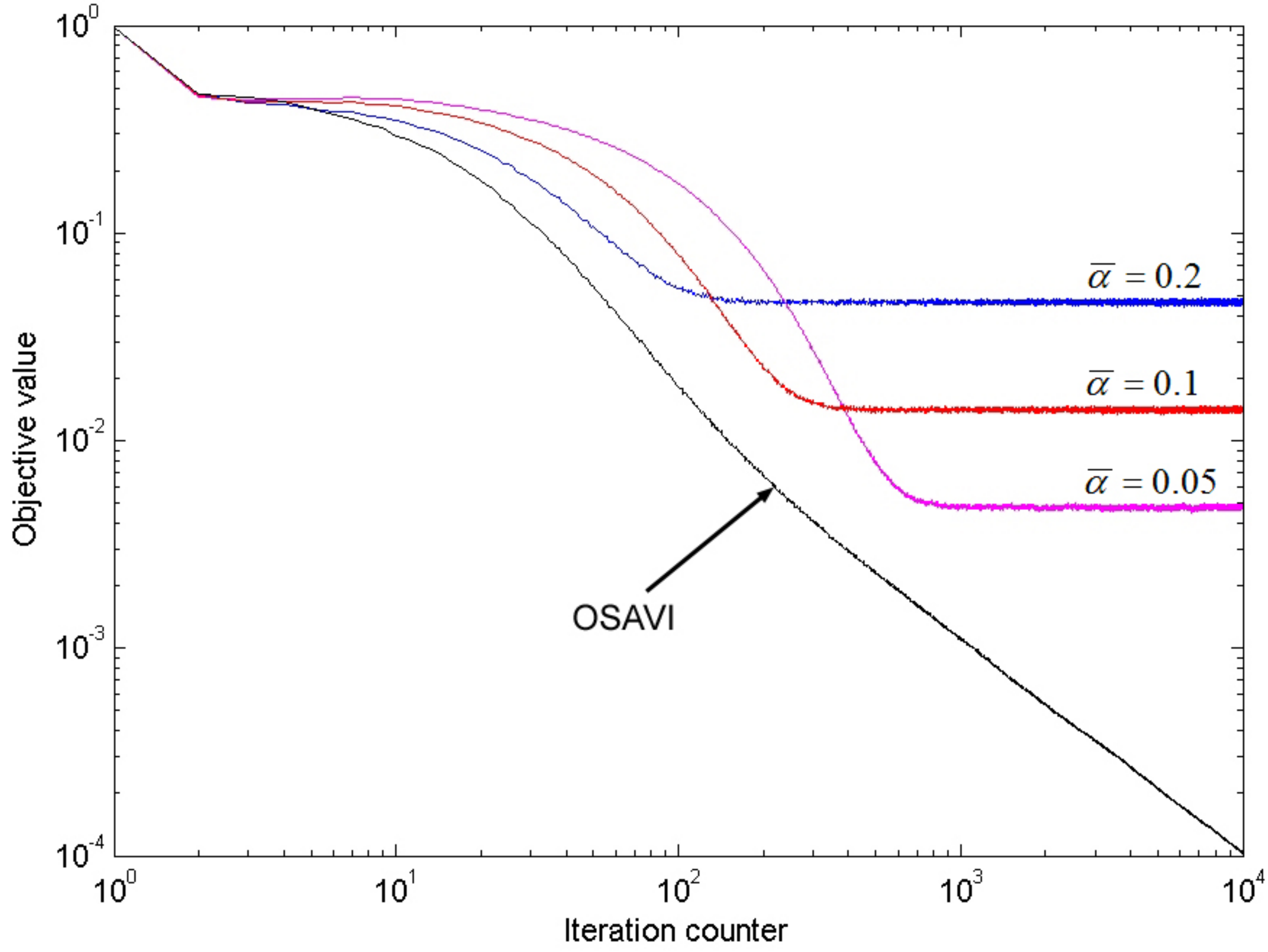}
\label{fig:mcctuning}
}
\subfigure[Harmonic rule.]{
\includegraphics[width=3.1in]{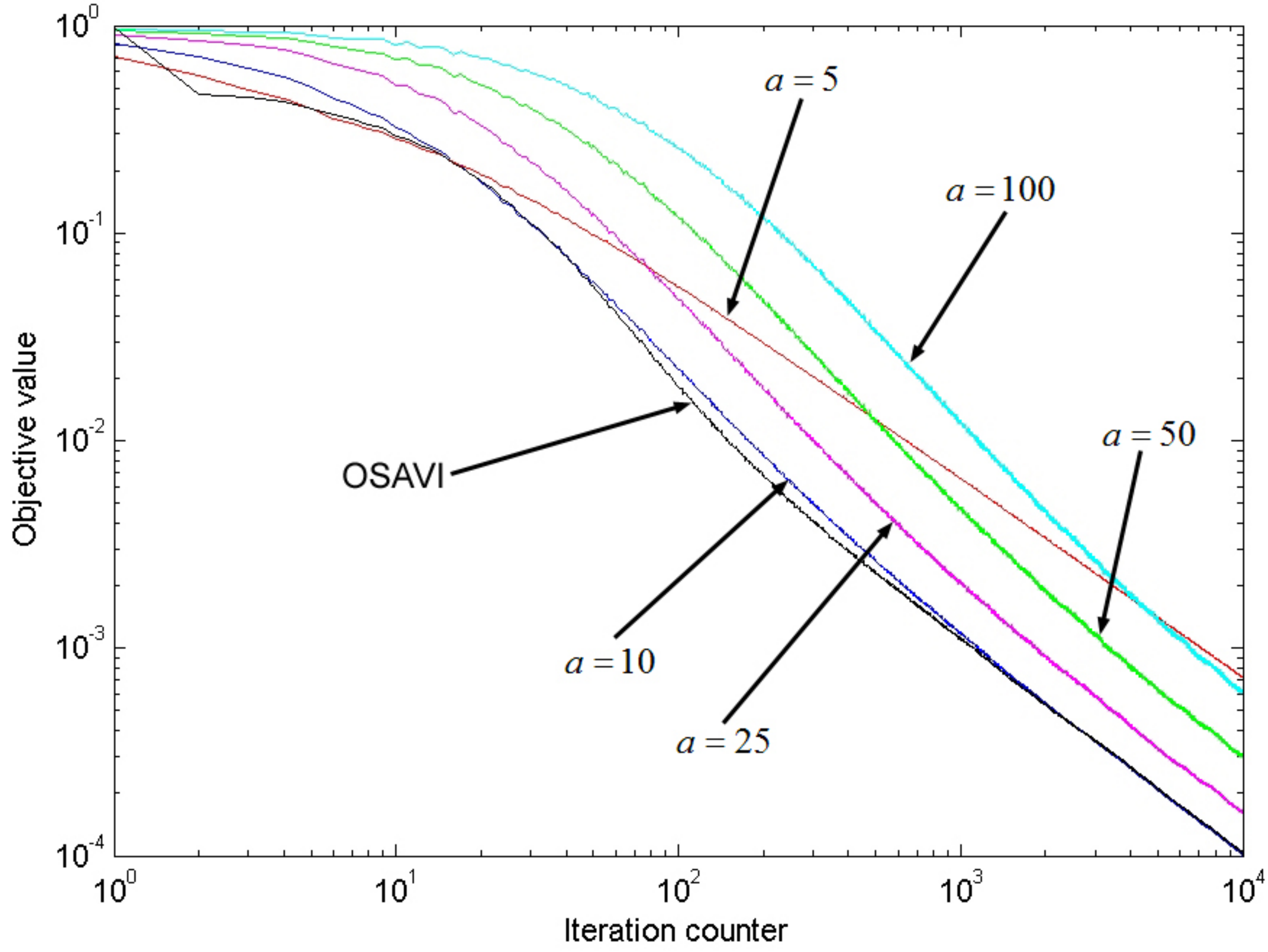}
\label{fig:hartuning}
}
\caption{Sensitivity of (a) McClain's rule and (b) the harmonic rule to their respective tunable parameters.}
\label{fig:mcchartuning}
\end{figure}

Figure \ref{fig:mcctuning} shows the sensitivity of McClain's rule to the choice of tunable parameter $\bar{\alpha}$. The effect is very similar to the effect of using different constant values of $\bar{\nu}_{n-1}$ in Figure \ref{fig:osatuning}. Smaller values of $\bar{\alpha}$ give better (more stable) late-horizon performance and worse (slower) early-horizon performance.

The harmonic rule is analyzed in Figure \ref{fig:hartuning}. We see that $a = 10$ is a good choice for this problem, with the particular parameter values (variance and discount factor) that we have chosen. Larger values of $a$ are consistently worse, and smaller values are only effective in the very early iterations. However, $a = 10$ yields very good performance, the best out of all the competing stepsize rules.

In fact, it is possible to tune the harmonic rule to perform competitively against OSAVI. However, the best value of $a$ is highly problem-dependent. Figure \ref{fig:harvar} shows that $a = 10$ continues to perform well when $\sigma^2$ is increased to $4$, and even achieves a slightly lower objective value than the approximate optimal rule in the later iterations, although OSAVI performs noticeably better in the early iterations. However, Figure \ref{fig:hardisc} shows that $a = 100$ becomes the best value when the discount factor $\gamma$ is increased to $0.99$. The optimal rule has not been retuned in Figure \ref{fig:hardiscvar}; all results shown are for $\nu_{n-1} = 0.2$. Interestingly, it appears that the optimal choice of $a$ is more sensitive to the discount factor than to the signal-to-noise ratio.

We conclude based on Figures \ref{fig:hartuning} and \ref{fig:hardiscvar} that the best choice of $a$ in the harmonic stepsize rule is very sensitive to the parameters of the problem, and that the best choice of $a$ for one problem setting can perform very poorly for a different problem. By contrast, Figure \ref{fig:opttuning} shows that OSAVI is relatively insensitive to its tunable parameter. A simple value of $\nu_{n-1} = 0.2$ yields good results in all of the settings considered. We claim that OSAVI is a robust alternative to several leading stepsize rules.

\begin{figure}[t]
\centering
\subfigure[$\sigma^2 = 4$.]{
\includegraphics[width=3.1in]{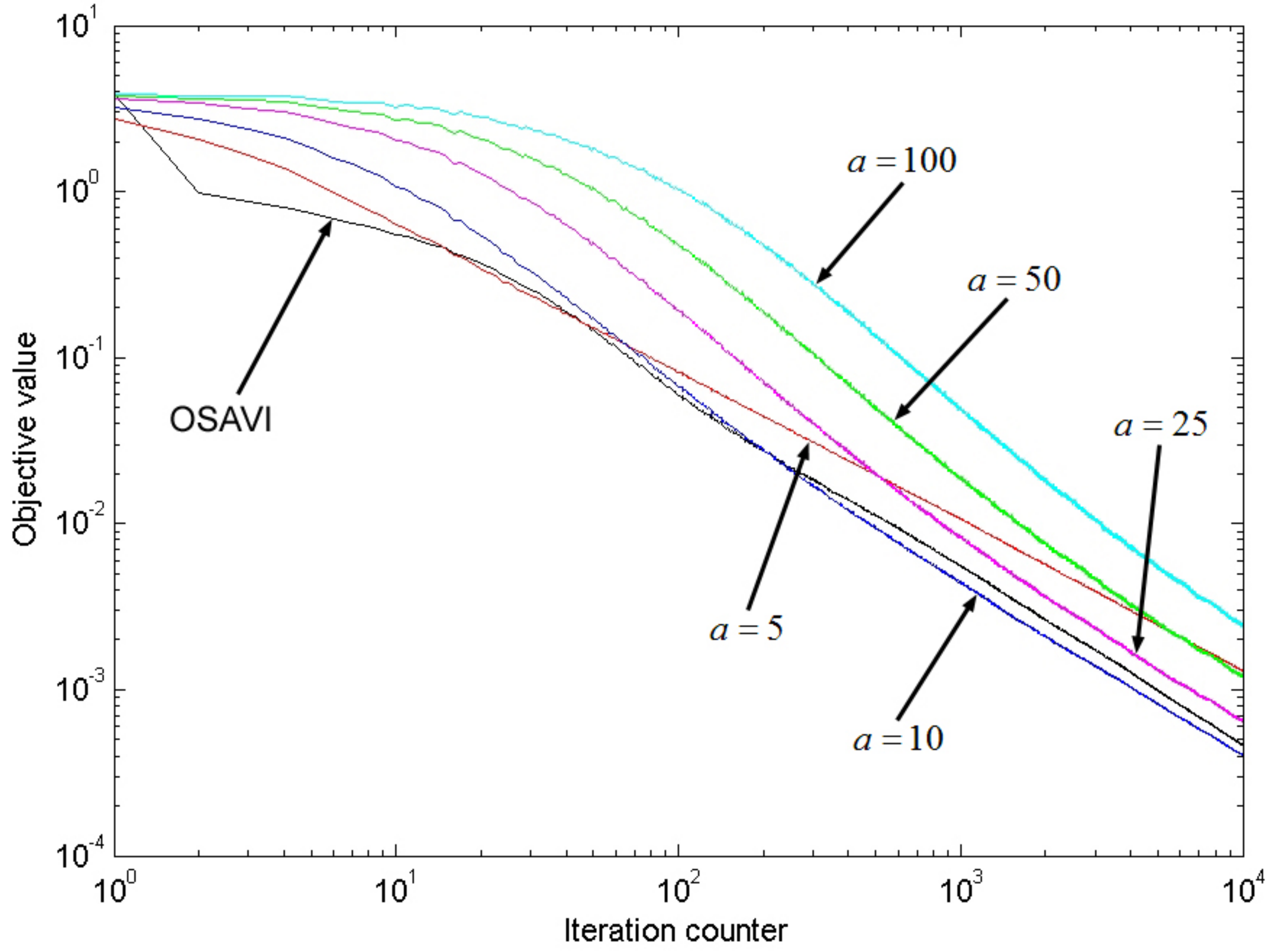}
\label{fig:harvar}
}
\subfigure[$\gamma = 0.99$.]{
\includegraphics[width=3.1in]{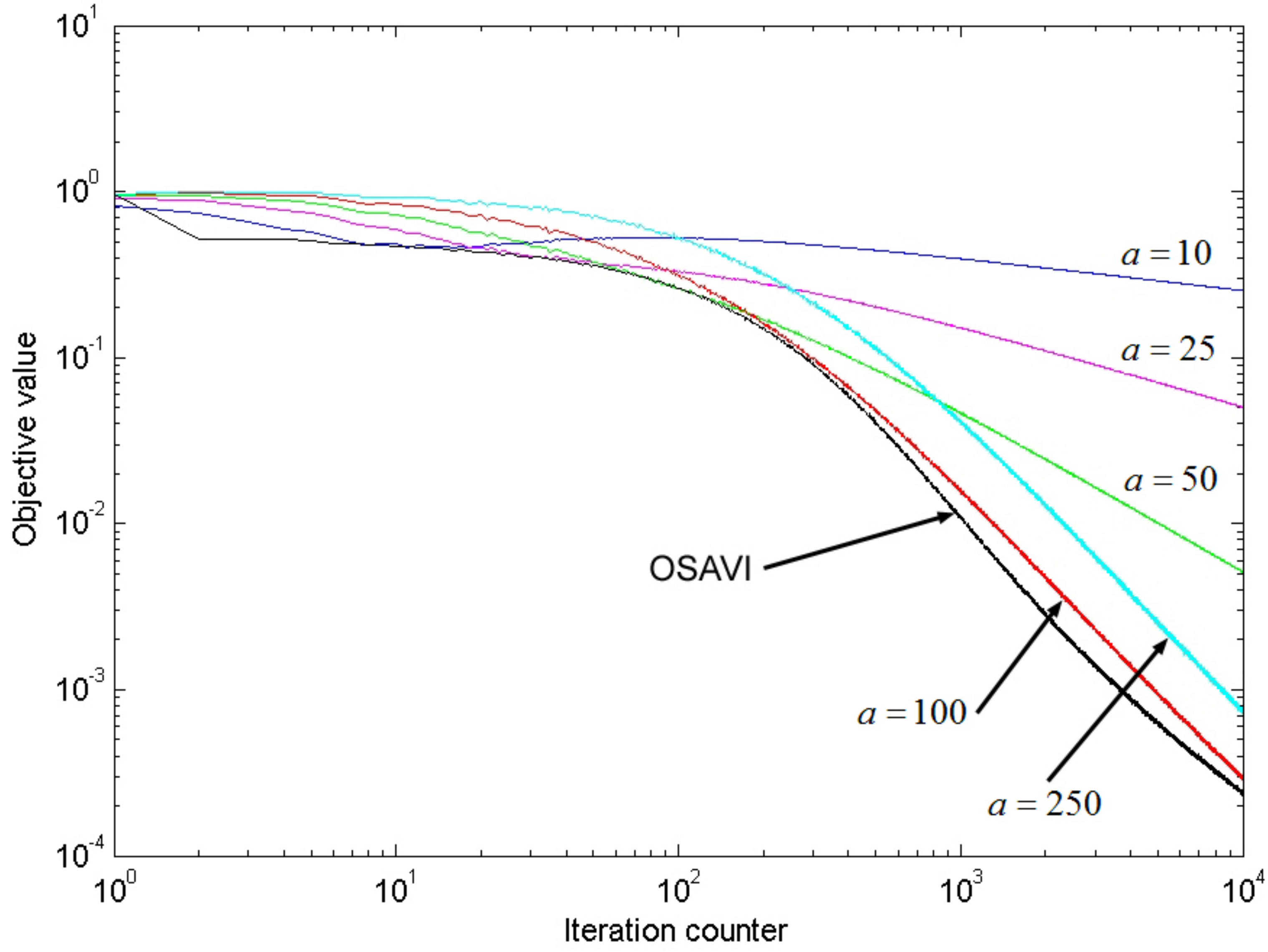}
\label{fig:hardisc}
}
\caption{Sensitivity of the harmonic stepsize rule in different problem settings.}
\label{fig:hardiscvar}
\end{figure}

\section{Experimental study: general MDP}\label{sec:ext}

We also tested the general OSAVI rule from Section \ref{sec:unknownc} on a synthetic MDP with $100$ states and $10$ actions per state, generated in the following manner. For state $S$ and action $x$, with probability $0.8$ the reward $C\left(S,x\right)$ was generated uniformly on $\left[0,2\right]$, and with probability $0.2$ it was generated uniformly on $\left[18,20\right]$. For each $\left(S,x\right)$, we randomly picked $10$ states to be reachable. For each such state $S'$, we generated a number $b_{S,S',x} \sim U\left[0,1\right]$ and let $\frac{b_{S,S',x}}{\sum_{S''} b_{S,S'',x}}$ be the probability of making a transition to $S'$ out of $\left(S,x\right)$. The transition probability to any state not reachable from $\left(S,x\right)$ was zero. In this manner, we obtained a sparse MDP with some high-value states, leading to some variety in the value function. We used value iteration to compute the true optimal value for $\gamma = 0.9$ and $\gamma = 0.99$.

\subsection{Infinite-horizon setting}\label{sec:extinf}

Each stepsize was implemented together with the following off-policy approximate value iteration algorithm. The value function approximation $\Vbar^0\left(S,x\right) = 0$ is defined for each state-action pair $\left(S,x\right)$, as in the Q-learning algorithm. Upon visiting state $S^n$, an action $x^n$ is chosen uniformly at random. Then, a new state $S'$ is simulated from the transition probabilities of the MDP. We then compute
\begin{eqnarray*}
\vhat^n &=& \max_x C\left(S',x\right) + \gamma \Vbar^{n-1}\left(S',x\right),\\
\Vbar^n\left(S^n,x^n\right) &=& \left(1-\alpha_{n-1}\right)\Vbar^{n-1}\left(S^n,x^n\right) + \alpha_{n-1}\vhat^n,
\end{eqnarray*}
where $\alpha_{n-1}$ is chosen according to some stepsize rule. The next state $S^{n+1}$ to be visited in the next iteration is then chosen uniformly at random (not set equal to $S'$).

We briefly discuss the reasoning behind this design. Any policy that uses the value function approximation to make decisions will also implicitly depend on the stepsize used to update that approximation. The stepsize affects the policy, which then affects the sequence (and frequency) of visited states, which in turn affects the calculation of future stepsizes. Ensuring that ``good'' states are visited sufficiently often is very important to the practical performance of ADP algorithms, but this issue (known as the problem of exploration) is quite separate from the problem of stepsize selection, and is outside the scope of our paper. We have sought to decouple the stepsize from the ADP policy by randomly generating states and actions.

We ran the above algorithm for $10^4$ iterations with each of the five stepsize rules from Section \ref{sec:mainexp}. Performance after $N$ iterations can be evaluated as follows. We find the policy $\pi$ that takes the action $\arg\max_x C\left(S,x\right) + \gamma\bar{V}^N\left(S,x\right)$ at state $S$, and then calculate the value $V^{\pi} = \left(I - \gamma P^{\pi}\right)^{-1}C^{\pi}$, where $P^{\pi}_{S,S'}$ is the probability of transitioning from $S$ to $S'$ under the policy $\pi$, and $C^{\pi}(S)$ is the reward obtained by following $\pi$ in state $S$. Then, we calculate
\begin{equation*}
\frac{1}{\left|\mathcal{S}\right|} \sum_S V^*(S) - V^{\pi}(S),
\end{equation*}
where $V^*$ is the true value function obtained from value iteration. This gives us the suboptimality of policy $\pi$. We average this quantity over $10^4$ simulations, each consisting of $N$ iterations of the learning algorithm. With $10^4$ simulations, the standard errors of this performance measure are negligible relative to its magnitude, and are omitted from the subsequent figures and discussion.

\begin{figure}[t]
\centering
\subfigure[$\gamma = 0.9$.]{
\includegraphics[width=3.1in]{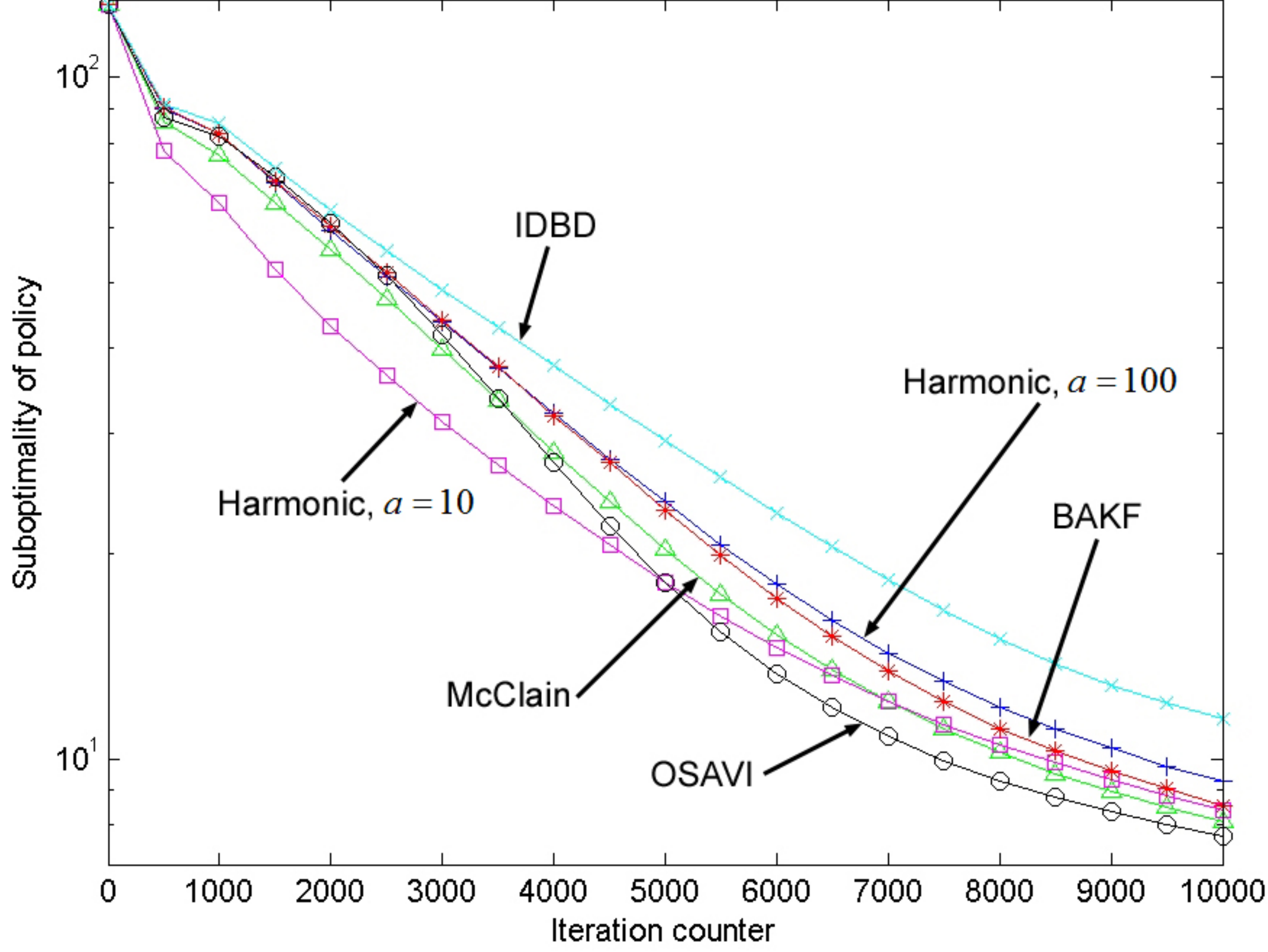}
\label{fig:MDP09}
}
\subfigure[$\gamma = 0.99$.]{
\includegraphics[width=3.1in]{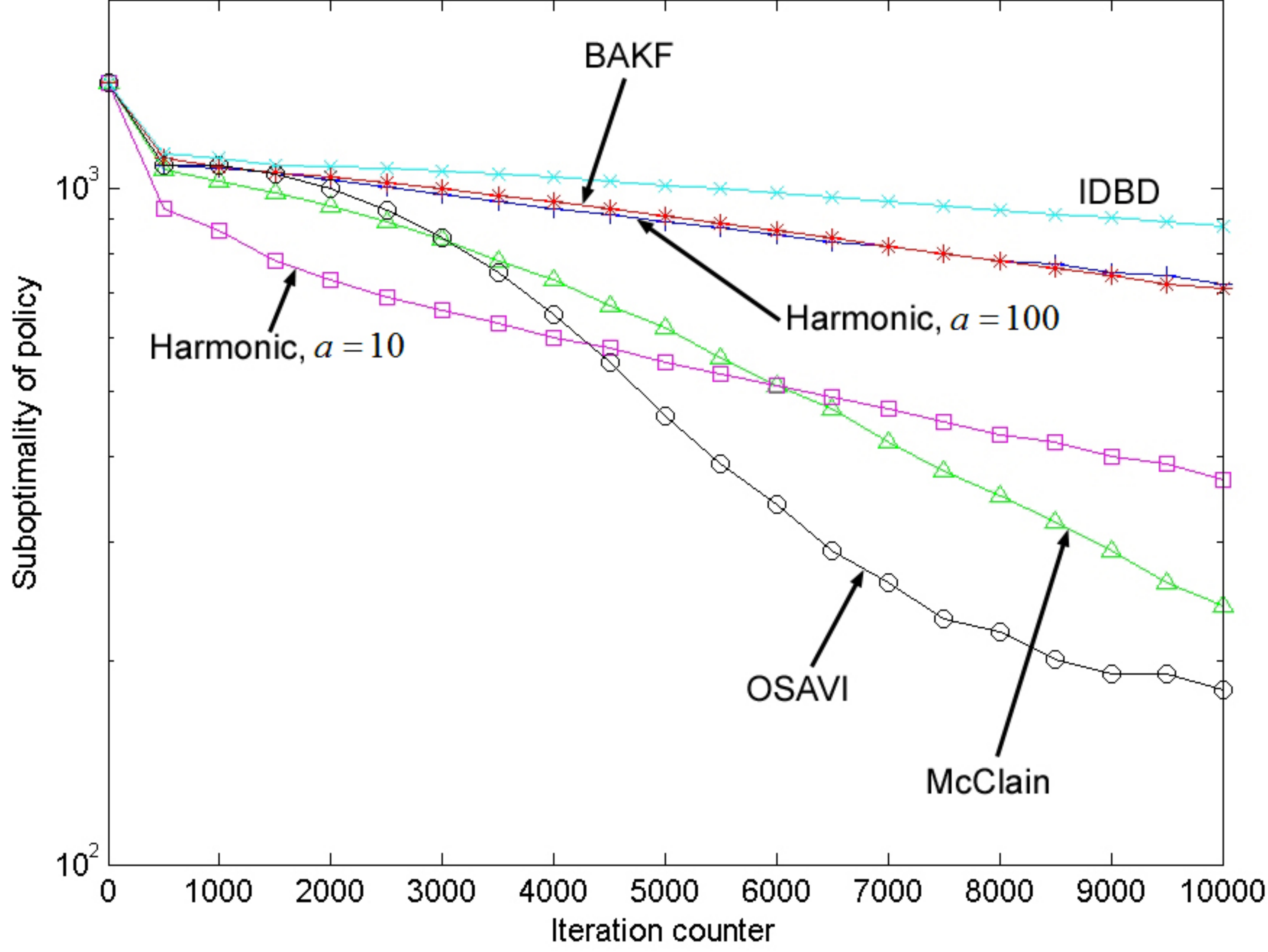}
\label{fig:MDP099}
}
\caption{Suboptimality for each stepsize rule for $\gamma = 0.9$ and $\gamma=0.99$ in the 100-state MDP.}
\label{fig:MDP}
\end{figure}

We compared the following stepsize rules: McClain's rule with $\bar{\alpha} = 0.1$, the harmonic rule with $a = 10$ and $a = 100$ (these are the tuned values that were found in Section \ref{sec:tunable} to work best for $\gamma = 0.9$ and $\gamma = 0.99$, respectively), the BAKF rule of \cite{GePo06} with a secondary stepsize of $0.05$, and OSAVI with a secondary stepsize of $0.2$. For IDBD, we experimented with several orders of magnitudes for $\theta$, and found that $\theta = 0.001$ produced good performance, although the difference between magnitudes was relatively small. We also made the stepsizes state-dependent in order to achieve quicker convergence. For example, the harmonic rule is given by $\alpha_{n-1}\left(S\right) = \frac{a}{a+N^n(S)}$ where $N^n(S)$ is the number of times state $S$ was visited in $n$ iterations. The parameters $\delta^n$, $\lambda^n$ and $\zeta^n$ used by OSAVI and BAKF were also chosen to be state-dependent.

Figure \ref{fig:MDP} shows the average suboptimality achieved by each stepsize rule over time, up to $10^4$ iterations. We see that, for both discount factors, the harmonic rule with $a = 10$ achieves the best performance early on, but slows down considerably in later iterations. OSAVI achieves the best performance in the second half of the time horizon, and the margin of victory is more clearly pronounced for $\gamma =0.99$.

We conclude that OSAVI yields generally competitive performance. The harmonic rule can be tuned to perform well, but performance is quite sensitive to the value of $a$, which is particularly visible in Figure \ref{fig:MDP099}. The secondary parameter for OSAVI was not tuned at all, as we wish to observe that a single constant value is sufficient to produce competitive performance.

\subsection{Finite-horizon setting}\label{sec:extfin}

Finite-horizon problems introduce the dimension that the relative size of the learning bias versus the noise in the contribution depends on the time period. As a result, the optimal stepsize behaviour changes with time. In the finite-horizon case, we run the same off-policy algorithm as before, with the update now calculated via the equations
\begin{eqnarray*}
\vhat^n_t &=& \max_x C\left(S'_t,x\right) + \gamma \Vbar^{n-1}_{t+1}\left(S',x\right)\\
\Vbar^n_t\left(S^n_t,x^n_t\right) &=& \left(1- \alpha_{n-1,t}\right)\Vbar^{n-1}_t\left(s^n_t,x^n_t\right) + \alpha_{n-1}\vhat^n_t
\end{eqnarray*}
for $t = 1,...,T-1$. The true value $V_t\left(S\right)$ of being in state $S$ at time $t$ can be found using backward dynamic programming; in the following, we use the values at $t = 1$ to evaluate all policies. Our performance measure is again the suboptimality of the policy induced by the value function approximation, averaged over all states. We used the same MDP as in Section \ref{sec:extinf} with the horizon $T = 20$, and the discount factor $\gamma =0.99$.

We compared the approximate version of the finite-horizon OSAVI rule from (\ref{eq:optfin}) to McClain's rule with $\bar{\alpha}=0.1$ and the harmonic rule $\alpha_{n-1,t} = \frac{a}{a+n}$ with $a = 10$ and $a = 100$. These rules achieved the best performance in the previous experiments, and can be easily applied to a finite-horizon problem. As before, all stepsizes were made to be state-dependent. Figure \ref{fig:MDPfinite} shows the average suboptimality of each stepsize rule. We see that the harmonic rule is competitive with OSAVI overall. However, the performance of $a = 10$ slows down in later iterations, as in Figure \ref{fig:MDP}. Furthermore, while $a=100$ outperforms OSAVI in the mid- to late iterations, OSAVI has largely closed the gap by the end and continues to improve, while the harmonic rule again slows down.

\begin{figure}[t]
\centering
\subfigure[Average suboptimality.]{
\includegraphics[width=3.1in]{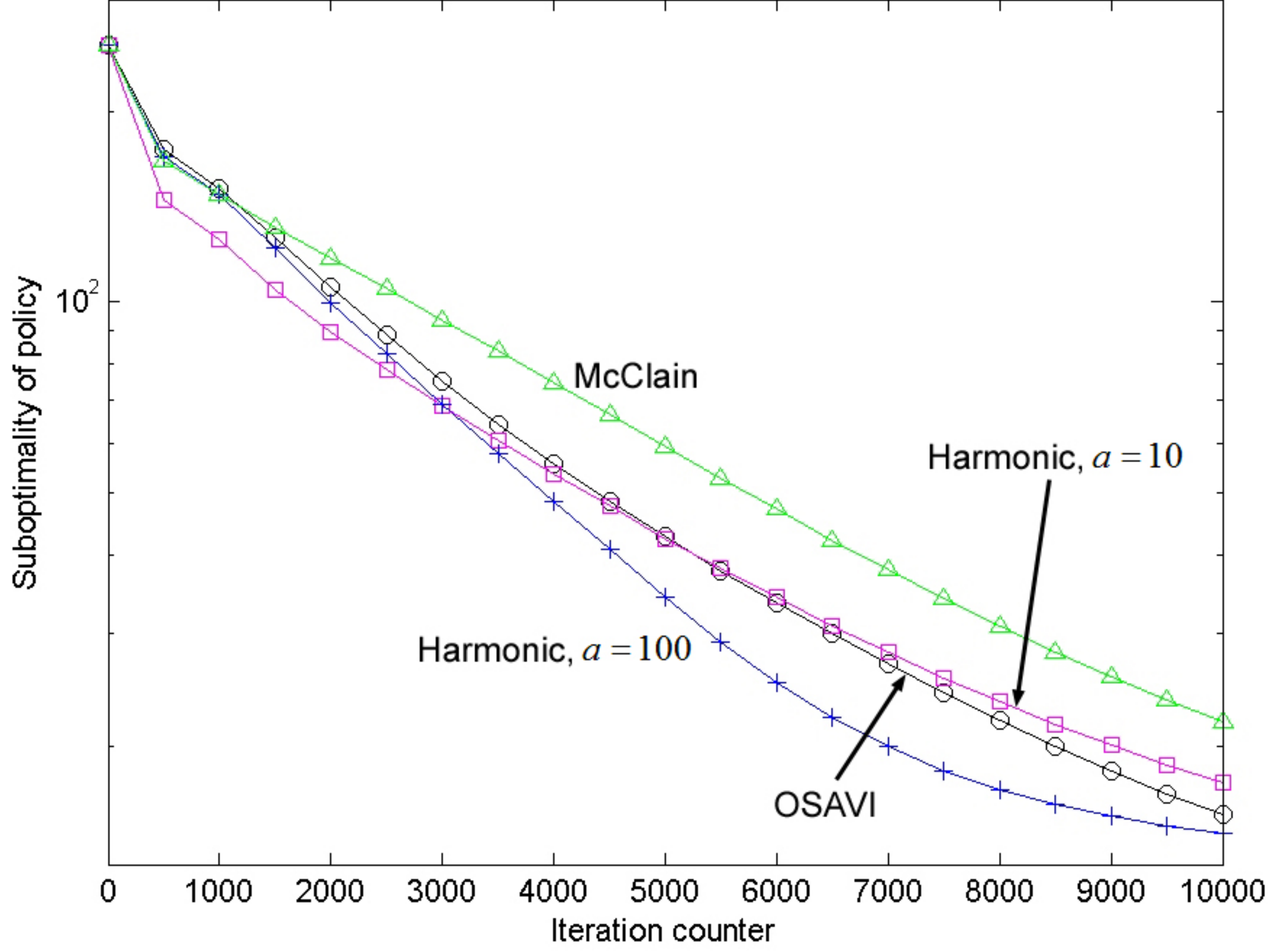}
\label{fig:MDPfinite}
}
\subfigure[Values of $\alpha_{n-1,t}$.]{
\includegraphics[width=3.1in]{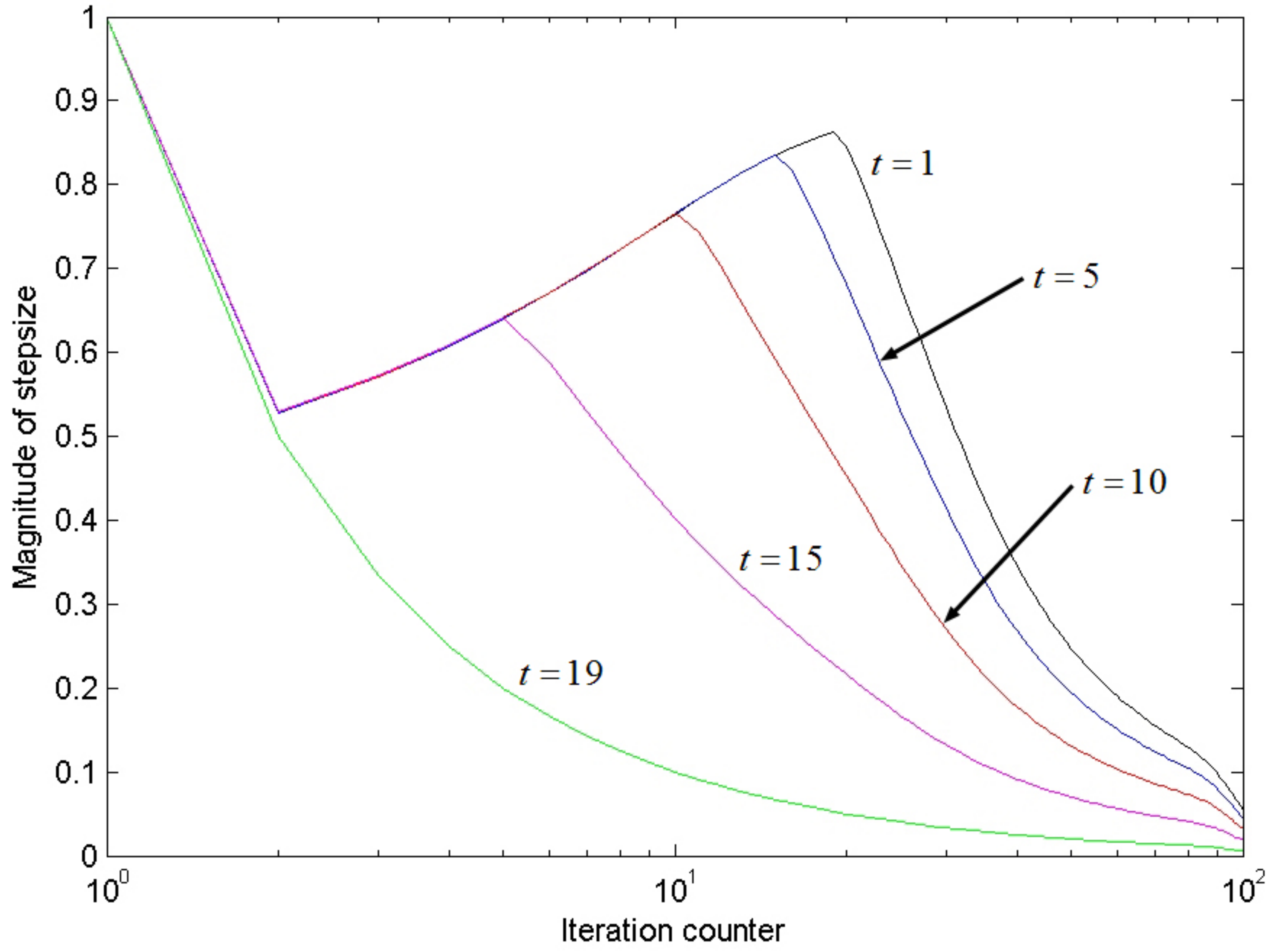}
\label{fig:MDPsizes}
}
\caption{Finite-horizon results: (a) Suboptimality for different stepsize rules. (b) Magnitudes of $\alpha_{n-1,t}$ for finite-horizon OSAVI.}
\label{fig:finiteMDP}
\end{figure}

Finally, Figure \ref{fig:MDPsizes} shows the magnitude of the stepsize $\alpha_{n-1,t}$ produced by the OSAVI formula in a simple synthetic MDP where all $100$ states are reachable from each $(S,x)$ and transition probabilities are normalized i.i.d. samples from a $U\left[0,1\right]$ distribution. Our purpose here is to illustrate the behaviour of the optimal stepsize for different $t$. When $t = 19$, OSAVI is identical to the $1/n$ rule, as we would expect, since this is the last time in the horizon. We assume $V_{20}\left(S\right) = 0$ for all $S$, so the observations $\vhat^n_{19}$ are stationary, and Corollary \ref{corol:stat} applies. For values of $t$ earlier in the time horizon, the optimal stepsize steadily increases, with the largest values of $\alpha_{n-1,t}$ being for $t = 1$. It takes a long time for our observations to propagate backward across the time horizon, and so we need larger stepsizes at time $t = 1$ to ensure that these observations have an effect. We note that, for earlier time periods, OSAVI goes through a period of exploration before settling on a curve which can be closely approximated by $\frac{a}{a+n}$ for a suitably calibrated choice of $a$. For the finite-horizon problem, $a$ should be different for each time period.

\section{Experimental study: ADP for a continuous inventory problem}\label{sec:storage}

The last part of our experimental study demonstrates how OSAVI can be used in conjunction with ADP on a problem where the state space is continuous. We present a stylized inventory problem where a generic resource can be bought and sold on the spot market, and held in inventory in the interim. The basic structure of our problem appears in applications in finance \citep{NaPo10},  energy \citep{LoMi10}, inventory control \citep{simao2009}, and water reservoir management \citep{Cervellera2006a}. We deliberately abstract ourselves from any particular setting, as we wish to keep the focus on the stepsize rule and test it in a generic setting for which ADP is required.

The state variable of the generic inventory problem contains two dimensions. Let $S_t = \left(R_t,P_t\right)$, where $R_t$ denotes the amount of resource currently held in inventory, and $P_t$ denotes the current spot price of the resource. We assume that $R_t$ can take values in the set $\left\{0,0.02,0.04,...,1\right\}$, representing a percentage of the total inventory capacity $\bar{R}$.  The action $x_t$ represents our decision to buy more inventory (positive values) or sell from our current stock (negative values). We assume that we can buy or sell up to $50\%$ of the total capacity in one time step, again in increments of $2\%$. Thus, there are up to $50$ actions in the problem. The reward $C\left(P,x\right) = - P\cdot \bar{R}\cdot x$ represents the revenue obtained (or cost incurred) after making decision $x$ given a price $P$.

While the resource variable $R_t$ is discrete, we assume that the spot price $P_t$ is continuous, and follows a geometric Ornstein-Uhlenbeck (mean reverting) process, a standard price model in finance and other areas. With minor modifications to the problem, $P_t$ could also be changed into an exogenous supply process, which we could draw from to satisfy a demand. The important aspect is that $P_t$ is continuous, which makes it impossible to solve (\ref{eq:valueiteration}) for every state. Furthermore, even for a given state $S_t$, computing the expectation in (\ref{eq:valueiteration}) is difficult, because the transition to the next state $S_{t+1}$ depends on a continuous random variable. For these reasons, we approach the problem using approximate dynamic programming with a discrete value function approximation. To address the issue of the continuous transition to the next state, we use the post-decision state concept introduced in \cite{VaBeLeTs97} and discussed extensively by \cite{Po11}. Given a state $S_t$ and a decision $x_t$, the post-decision state $S^x_t = \left(R^x_t,P^x_t\right)$ is given by the equations
\begin{eqnarray*}
R^x_t &=& R_t + x_t\mbox{,}\\
P^x_t &=& P_t\mbox{.}
\end{eqnarray*}
The next pre-decision state $S_{t+1}$ is then obtained by setting $R_{t+1} = R^x_t$ and simulating $P_{t+1}$ from the price process. Given a value function approximation $\Vbar^{n-1}$, the update $\vhat^n_t$ is computed using
\begin{equation*}
\vhat^n_t = \max_{x_t} C\left(P^n_t,x_t\right) + \gamma\Vbar^{n-1}_t\left(S^{x,n}_t\right).
\end{equation*}
This quantity is then used to update the previous post-decision state, that is,
\begin{equation*}
\Vbar^n_{t-1}\left(S^{x,n}_{t-1}\right) = \left(1-\alpha_{n-1,t-1}\right) \Vbar^{n-1}_{t-1}\left(S^{x,n}_{t-1}\right) + \alpha_{n-1,t-1}\vhat^n_t\mbox{.}
\end{equation*}
Thus, we can adaptively improve our value function approximation without computing an expectation.

For our value function approximation, we used a lookup table where the log-price $\log P_t$ was discretized into $34$ intervals of width $0.125$ between $-2$ and $2$. Thus, the table contained a total of $51 \cdot 34 = 1734$ entries, with each entry initialized to a large value of $10^4$, in keeping with the recommendation in Sec. 4.9.1 of \cite{Po11} to use optimistic initial estimates. However, while the approximation used a discretized state space, our experiments simulated $P_t$ using the continuous price process. The price was only discretized during calls to the lookup table. This is an important detail: while we use a discrete value function approximation, we are still solving the original continuous problem.

The price process was instantiated with $P_0 = 30$, mean-reversion parameter $0.0633$ and volatility $0.2$. Most prices are thus around $\$30$, but sharp spikes are possible. As before, we use a pure exploration policy where each action $x_t$ is chosen uniformly at random. Also as before, we simulate a future state from the price process in order to compute $\vhat^n_t$, but the next state to actually be visited by the algorithm is generated randomly (the resource level is generated uniformly at random, and the log-price is generated uniformly between $-2.125$ and $2.125$). Recall that this is necessary in order to separate the performance of the stepsize from the quality and architecture of the value function approximation.

We used the same policies as in Section \ref{sec:ext}: McClain's rule with $\bar{\alpha} = 0.1$, the harmonic rule with $a = 10$, the BAKF rule of \cite{GePo06} with a secondary stepsize of $0.05$, IDBD with $\theta = 0.001$, and OSAVI with a secondary stepsize of $0.2$. To evaluate the performance of each stepsize rule after $N$ iterations, we fixed the approximation $\Vbar^N$ and then simulated the total reward obtained by making decisions of the form $x_t = \arg\max_x C\left(P_t,x\right) + \gamma V^N_t\left(S^x_t\right)$ in both finite- and infinite-horizon settings. This quantity was averaged over $2.5 \times 10^4$ sample paths. Figure \ref{fig:energy} reports the performance of the approximation obtained using different stepsize rules. Since our objective is to maximize revenue, higher numbers on the $y$-axis represent better quality.

\begin{figure}[t]
\centering
\subfigure[Infinite horizon.]{
\includegraphics[width=3.1in]{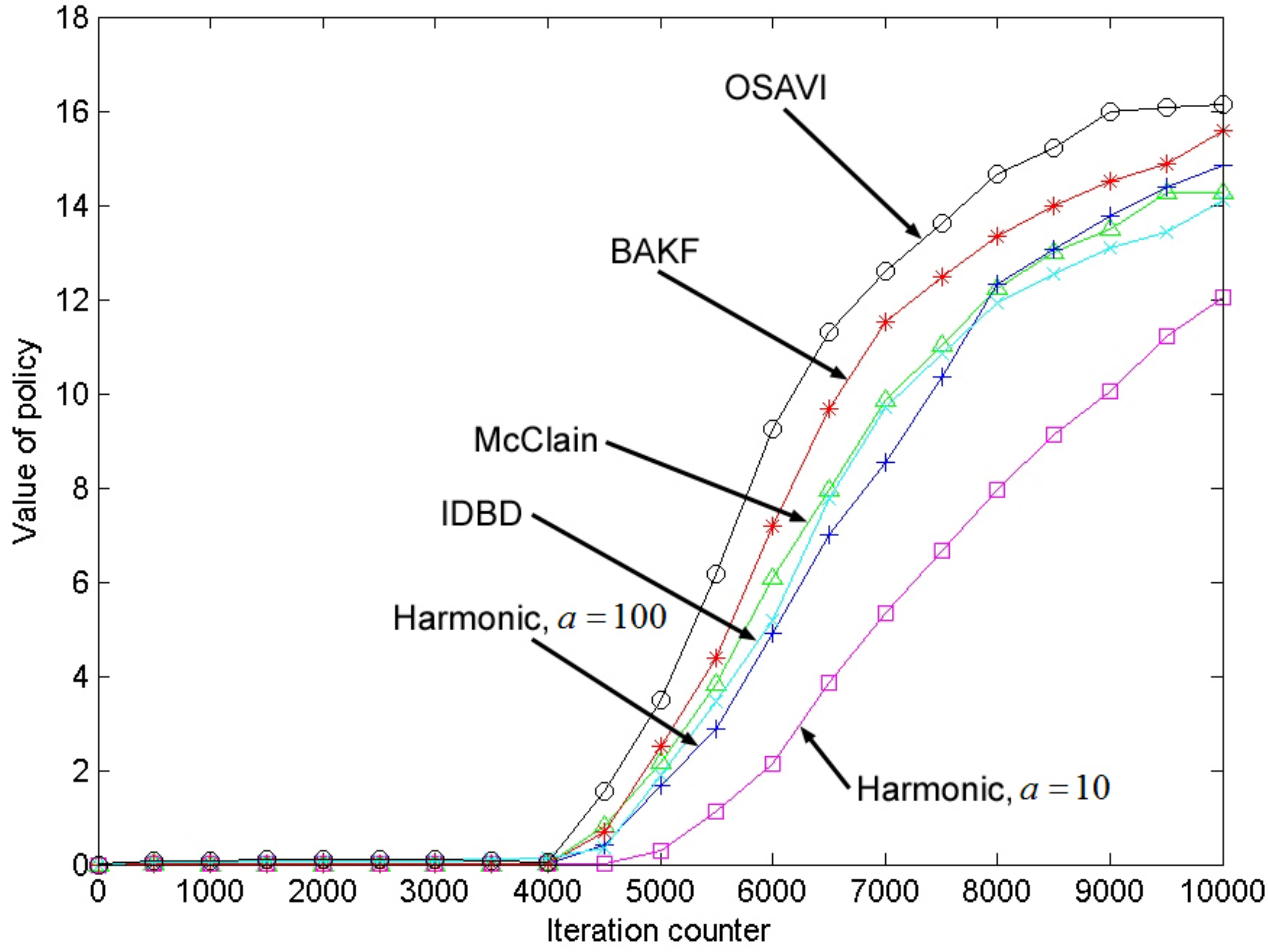}
\label{fig:energyinfinite}
}
\subfigure[Finite horizon.]{
\includegraphics[width=3.1in]{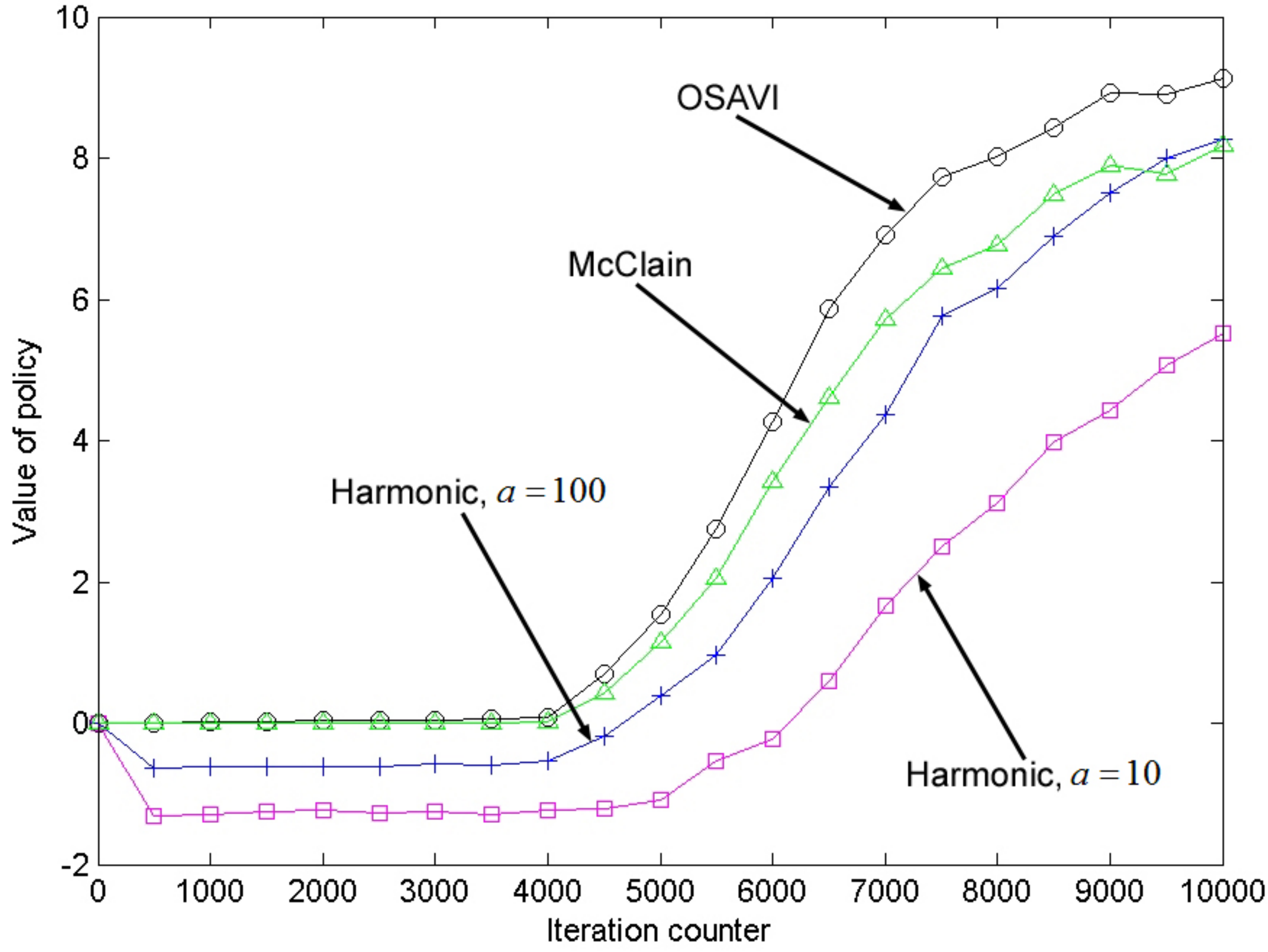}
\label{fig:energyfinite}
}
\caption{Offline policy values for different stepsize rules in the inventory problem with (a) infinite horizon and (b) finite horizon.}
\label{fig:energy}
\end{figure}

Figure \ref{fig:energyinfinite} shows the performance of OSAVI in the infinite-horizon setting. Because of the larger size of the inventory problem, we require several thousand iterations in order to obtain any improvement in the target policy specified by $\Vbar^N$. After $4000$ iterations, we find that OSAVI consistently yields the most improvement in the value of the target policy. Analogously to Figure \ref{fig:MDPfinite} in Section \ref{sec:extfin}, we also compared OSAVI to the harmonic rule in a finite-horizon setting; the results are shown in Figure \ref{fig:energyfinite}. As in the infinite-horizon setting, several thousand iterations are required before any improvement can be observed, but OSAVI consistently outperforms the best version of harmonic.

Our experiments on the inventory problem offer additional evidence that our new stepsize rule can be applicable to more complex dynamic programming problems, which cannot be solved exactly, and where additional techniques such as the post-decision state variable are necessary to deal with continuous state spaces and difficult expectations. Even in the streamlined form considered here, the inventory problem features a continuous price variable, and the value of being in a state depends on the behavior of a mean-reverting stochastic differential equation. The fact that OSAVI retains its advantages over other stepsize rules in this setting is an encouraging sign.

\section{Conclusion}\label{sec:conclusion}

We have proposed a mathematical framework for analyzing stepsize selection in approximate dynamic programming. Our analysis is based on a stylized model of a single-state, single-action MDP. We used this model to derive new rate of convergence results for the popular $1/n$ stepsize rule. Even in this stylized problem, approximate value iteration converges so slowly under the $1/n$ rule as to be virtually unusable for most infinite-horizon applications. This underscores the importance of stepsize selection in general dynamic programming problems.

We have derived a new optimal stepsize minimizing the prediction error of the value function approximation in the single-state model. To our knowledge, this stepsize is the first to take into account the covariance between the observation we make of the value of being in a state, and our approximation of that value, a property that is inherent in approximate value iteration. Furthermore, we are able to compute a closed-form expression for the prediction bias in the single-state, single-action case, considerably simplifying the task of estimating this quantity in the general case. The rule can be easily extended to a general MDP setting, both finite- and infinite-horizon.

We have tested our stepsize rule against several leading deterministic and stochastic rules. In the single-state, single-action case, we consistently outperform the other stepsize rules. While some competing rules (particularly the harmonic rule) can be tuned to yield very competitive performance, they are also very sensitive to the choice of tuning parameter. On the other hand, our stepsize rule is robust, displaying little sensitivity to the parameter used to estimate the one-period reward. We also tested our stepsize rule on a general discrete-state MDP, as well as on a more complex ADP problem. We found that OSAVI performs competitively against the other rules in both finite- and infinite-horizon settings.

We conclude that our stepsize rule can be a good alternative to other leading stepsizes. Our conclusion reflects the particular set of experiments that we chose to run. It is important to remember that deterministic stepsizes such as the harmonic rule can be finely tuned to a particular problem, resulting in better performance than the adaptive rule that we present. The strength of our rule, however, is its ability to adjust to the evolution of the value function approximation, as well as its relative lack of sensitivity to tuning.

\appendix
\section{Slow convergence of $\alpha_{n-1} = 1/n$}

Setting the stepsize to $\alpha_{n} = \frac{1}{n+1}$, we know \citep[e.g. from][]{KuYi97} that the approximation is guaranteed to converge to the optimal value. Let $\vbar^0 = 0$ be the initial approximation. We rewrite (\ref{eq:vbar}) for time $n+1$ using (\ref{eq:vhat}) as
\begin{equation}
\vbar^{n+1} - \vbar^n = \alpha_n(\vhat^{n+1}-\vbar^n) = \frac{1}{n+1}(c - (1-\gamma)\vbar^n),
\label{eq:VbarDifference}
\end{equation}
where the above equations hold for $n\in\mathds{N}^* = \{0,1,2,\ldots\}$. We characterize the slow convergence of approximate value iteration smoothed with a $1/n$ stepsize by bounding $\vbar^n$ above and below by
\begin{equation}
    1 - (n+1)^{-(1-\gamma)}
    \le \frac{\vbar^n}{v^*}
    \le 1 - \frac{\gamma^2 + \gamma - 1}{\gamma}n^{-(1-\gamma)}
        - \frac{1-\gamma}{\gamma}\frac{1}{n}
\end{equation}
for $n\ge1$.
This bound implies that $\vbar^n$ converges particularly slowly for $\gamma$ larger than $.8$.  For example, for $\gamma\ge.9$, this bound tells us that approximate value iteration takes at least $10^{19}$ iterations to reach within $1\%$ of optimal.  This is too many iterations for even the fastest implementation.

We approximate the discrete time update equation (\ref{eq:VbarDifference}) with a continuous time differential equation in which $\vbar^{n+1}-\vbar^n$ is approximated by the derivative of $\vbar^n$ with respect to $n$.  The first step is to extend the definition of $\vbar^n$ from the natural numbers onto the positive reals through a piecewise linear interpolation. We define
\begin{eqnarray}
\vbar(n) &=& (1-(n-\lfloor n \rfloor))\vbar^{\lfloor n \rfloor} + (n-\lfloor n \rfloor)\vbar^{\lceil n \rceil}\label{eq:vbarextensionavg}\\
&=& \vbar^{\lfloor n \rfloor} + (n-\lfloor n \rfloor)(\vbar^{\lceil n \rceil}-\vbar^{\lfloor n \rfloor})
\label{eq:VbarExtension}
\end{eqnarray}
for all $n\in\Rp$, where $\vbar(n)$ is given for $n\in\mathds{N}^*$ by the recursion defined by (\ref{eq:vbar}). Here, $\lfloor n \rfloor$ is the greatest integer less than or equal to $n$, and $\lceil n \rceil$ is the least integer greater than or equal to $n$. As can be seen in (\ref{eq:vbarextensionavg}), we are simply writing $\vbar\left(n\right)$ as a weighted average of its rounded values. Observe that, for $n$ integer, $\vbar(n) = \vbar^n$.

First, we note some well-known properties of the sequence $\{\vbar(n)\}$ in Lemma \ref{lemma:vIncreasingConcaveBounded}.  The proof is straightforward, and we omit it.

\begin{lemma}
\label{lemma:vIncreasingConcaveBounded}
$\vbar(n)$ is increasing and concave in $n$, and bounded above by $\frac{c}{1-\gamma}$.
\end{lemma}

The right derivative of the interpolation function $\vbar$ is given by
\begin{equation}
\dvp(n) = \frac{1}{\lfloor n \rfloor+1}(c-(1-\gamma)\vbar(\floor{n})). \label{eq:dvplus}
\end{equation}
Our strategy is to construct functions $U:\R^2\to\R$ and $L:\R^2\to\R$ such that $L(\vbar(n),n) \le \dvp(n) \le U(\vbar(n),n)$.  Fix any $n_0\in\Rp$.  Then a lower bound for $\vbar$ on $\left[ n_0, \infty\right)$ is given by any solution $l$ to the differential equation
\begin{equation*}
l'(n) = L(l(n),n)
\end{equation*}
with boundary condition $l(n_0) = \vbar(n_0)$.  Similarly, an upper bound for $\vbar$ on $\left[ n_0, \infty\right)$ is given by any solution $u$ to the differential equation
\begin{equation*}
u'(n) = U(u(n),n)
\end{equation*}
with boundary condition $u(n_0) = \vbar(n_0)$.

\begin{lemma}
\label{lemma:BoundDvp}
For $n\in\mathds{N}^*$,
$\dvp(n) \le \frac{\vbar(n+1)}{n+1}$.
\end{lemma}

\begin{proof}
We begin by noting,
$\vbar(n+1) = \sum_{k=0}^n \vbar(k+1)  - \vbar(k) = \sum_{k=0}^n \dvp(k)$.
By the concavity of $\vbar$ as shown in Lemma \ref{lemma:vIncreasingConcaveBounded}, $\dvp(k) \le \dvp(n)$ for all $k\le n$, so $\vbar(n+1) \le \sum_{k=0}^n \dvp(n) = (n+1)\dvp(n)$.  Dividing by $n+1$ completes the proof.
\end{proof}

\begin{theorem}
\label{thm:UpperBound}
For any $n_0 > 0$, $\vbar$ is bounded above by
\begin{equation*}
\vbar(n) \le \frac{c}{1-\gamma}\left[ 1 - bn^{-(1-\gamma)} - \frac{1-\gamma}{\gamma}\frac{1}{n} \right]
\end{equation*}
where
\begin{equation}
b = n_0^{1-\gamma}\left[ 1 - \frac{1-\gamma}{n_0\gamma} - \frac{1-\gamma}{c}\vbar(n_0)\right]. \label{eq:b}
\end{equation}
\end{theorem}

\begin{proof}
We begin by rewriting $\vbar(\floor{n})$ as
\begin{eqnarray*}
\vbar(\floor{n})
    &=& \vbar(\floor{n}+1) - \left(\vbar(\floor{n}+1) - \vbar(\floor{n})\right) \\
    &=& \vbar(\floor{n}+1) - \dvp(\floor{n}) \\
    &\ge& \vbar(\floor{n}+1) - \frac{\vbar(\floor{n}+1)}{\floor{n}+1} \\
    &\ge& \vbar(\floor{n}+1) - \frac{c/(1-\gamma)}{\floor{n}+1}
\end{eqnarray*}
where the third step is by Lemma \ref{lemma:BoundDvp}, and the fourth step is by Lemma \ref{lemma:vIncreasingConcaveBounded}.  We combine this with (\ref{eq:dvplus}) to write
\begin{equation*}
\dvp(n) \le \frac{1}{\floor{n}+1}\left( c-(1-\gamma)\vbar(\floor{n}+1)+\frac{c}{\floor{n}+1}\right).
\end{equation*}
Then, the inequality $n \le \floor{n}+1$ implies that
\begin{equation*}
\dvp(n) \le \frac{1}{n} \left( c-(1-\gamma)\vbar(\floor{n}+1)+\frac{c}{n}\right).
\end{equation*}
The same inequality $n \le \floor{n} + 1$ together with Lemma \ref{lemma:vIncreasingConcaveBounded} imply that $\vbar(\floor{n} +1) \ge \vbar(n)$, which implies
\begin{equation*}
\dvp(n) \le \frac{1}{n} \left( c-(1-\gamma)\vbar(n)+\frac{c}{n}\right).
\end{equation*}
Defining $U:\R^2\to\R$ by $U(v,n) = \frac{1}{n} \left( c-(1-\gamma)v+\frac{c}{n}\right)$, $\dvp(n) \le U(v(n),n)$.  We solve the differential equation
\begin{equation}
u'(n) = U(u(n),n) \label{eq:UpperDiffEQ}
\end{equation}
with boundary condition
\begin{equation}
u(n_0) = \vbar(n_0).  \label{eq:UpperDiffEQBoundary}
\end{equation}
The solution, $u$, is an upper bound for $\vbar$ in the sense that $u(n) \ge \vbar(n)$ for all $n\ge n_0$.  We solve for $u$ using the general solution for first order linear differential equations \citep{BoDi97}.  The integrating factor $\mu(n)$ is the integral of the term multiplying $u(n)$,
\begin{equation*}
\mu(n) = \exp \left( \int (1-\gamma)m^{-1}\ dm \right) = \exp \left( (1-\gamma) \log(n) \right) = n^{1-\gamma}.
\end{equation*}
The integral of the right hand side multiplied by the integrating factor is
\begin{equation*}
\int \mu(m) c(m^{-1} + m^{-2})\ dm = c\int m^{-\gamma} + m^{-1-\gamma}\ dm = \frac{c}{1-\gamma} \left( n^{1-\gamma} - \frac{1-\gamma}{\gamma} n^{-\gamma} - b\right),
\end{equation*}
where $b$ is any scalar.
Thus, the solution of (\ref{eq:UpperDiffEQ}) is
\begin{eqnarray*}
u(n)
    &=& \frac{1}{\mu(n)}\int \mu(m) c(m^{-1} + m^{-2})\ dm \\
    &=& \frac{c}{1-\gamma} \frac{n^{1-\gamma} - \frac{1-\gamma}{\gamma} n^{-\gamma} - b}{n^{1-\gamma}} \\
    &=& \frac{c}{1-\gamma} \left( 1 - \frac{1-\gamma}{\gamma}\frac{1}{n} - bn^{-(1-\gamma)}\right),
\end{eqnarray*}
where $b$ is chosen to satisfy the boundary condition (\ref{eq:UpperDiffEQBoundary}).  Solving the relation
\begin{equation*}
\vbar(n_0) = u(n_0) = \frac{c}{1-\gamma} \left( 1 - \frac{1-\gamma}{\gamma}\frac{1}{n_0} - bn_0^{-(1-\gamma)}\right)
\end{equation*}
for $b$ gives (\ref{eq:b}).
\end{proof}

In Theorem \ref{cor:ChooseOne0}, the constant $b$ is strictly positive only for $\gamma > \frac{-1+\sqrt{5}}2 \approx .618$ since $\gamma b = \gamma^2 + \gamma - 1 = \left( \gamma - \frac{-1+\sqrt5}2\right) \left( \gamma - \frac{-1-\sqrt5}2\right)$.  When $b$ is strictly positive, Theorem \ref{cor:ChooseOne0} provides a useful bound on the asymptotic convergence of $\vbar$.  When $b$ is negative, however, as $n$ becomes large the upper bound
$\frac{c}{1-\gamma}\left[ 1 - bn^{-(1-\gamma)} - \frac{1-\gamma}{\gamma}\frac{1}{n} \right]$ becomes larger than the trivial upper bound $\frac{c}{1-\gamma}$ shown in Lemma \ref{lemma:vIncreasingConcaveBounded} and the bound is no longer useful.  To obtain useful bounds from Theorem \ref{thm:UpperBound} for a broader range of $\gamma$ we may increase the $n_0$ chosen.  Increasing $n_0$ also increases $b$ and tightens the bound across all $\gamma$.

\section{Proofs}

\subsection{Proof of Theorem \ref{eq:LowerBound0}}

We first generalize Theorem \ref{eq:LowerBound0} to
\begin{equation}
\vbar(n) \ge \frac{c}{1-\gamma}\left( 1 - (n+1)^{-(1-\gamma)}\right) \mbox{ for all } n\ge0.\label{eq:LowerBound}
\end{equation}
We begin rewriting (\ref{eq:dvplus}) using the inequality $\floor{n} \le n$ as
\begin{equation*}
\dvp(n) \ge \frac{1}{n+1}(c-(1-\gamma)\vbar(\floor{n})).
\end{equation*}
Then, the same inequality $\floor{n} \le n$ together with Lemma \ref{lemma:vIncreasingConcaveBounded} imply that $\vbar(\floor{n}) \le \vbar(n)$, which implies
\begin{equation*}
\dvp(n) \ge \frac{1}{n+1}(c-(1-\gamma)\vbar(n)).
\end{equation*}
Defining $L:\R^2\to\R$ by $L(v,n) = \frac{1}{n+1}(c-(1-\gamma)v)$, $\dvp(n) \ge L(v(n),n)$.  We solve the differential equation
\begin{equation*}
l'(n) = L(l(n),n) = \frac{1}{n+1}(c-(1-\gamma)l(n))
\end{equation*}
with boundary condition $l(0) = \vbar(0)$.

The solution to this differential equation satisfies $l(n) \le \vbar(n)$ for all $n\ge0$ and thus bounds $\vbar$ from below.  We solve for $l$ using the general solution for first order linear differential equations \citep{BoDi97}.  The integrating factor is
$\mu(n) = \exp\left[\int (1-\gamma)\frac{dm}{m+1}\right]= \exp\left[(1-\gamma) \log(n+1)\right]= (n+1)^{1-\gamma}$.
The solution $l$ is given by
\begin{eqnarray*}
l(n)&=& \frac{1}{\mu(n)} \int \mu(m) \frac{c}{m+1}\ dm \\
    &=& c(n+1)^{-(1-\gamma)}\int (m+1)^{1-\gamma} (m+1)^{-1}\ dm \\
    &=& c(n+1)^{-(1-\gamma)}\int (m+1)^{-\gamma}\ dm \\
    &=& c(n+1)^{-(1-\gamma)}\left( \frac{1}{1-\gamma}(n+1)^{1-\gamma} - b\right) \\
    &=& c\left( \frac{1}{1-\gamma} - b(n+1)^{-(1-\gamma)}\right),
\end{eqnarray*}
where $b$ is an integration constant chosen so that $l(0)=\vbar(0)=0$.  We plug in $n=0$ to this equation to see $0 = c\left( \frac{1}{1-\gamma} - b\right)$, implying that $b=\frac{1}{1-\gamma}$.  Thus,
\begin{equation*}
l(n)    = c\left( \frac{1}{1-\gamma} - \frac{1}{1-\gamma}(n+1)^{-(1-\gamma)}\right)
    = \frac{c}{1-\gamma}\left( 1 - (n+1)^{-(1-\gamma)}\right),
\end{equation*}
which completes the proof.

\subsection{Proof of Theorem \ref{cor:ChooseOne0}}

Substituting $n_0=1$ and $\vbar(1)=c$ into (\ref{eq:b}) gives
\begin{equation*}
b = 1 - \frac{1-\gamma}{\gamma} - \frac{1-\gamma}{c}c = - \frac{1-\gamma}{\gamma} + \gamma = \frac{\gamma^2 + \gamma - 1}{\gamma}
\end{equation*}
as required.

\subsection{Proof of Proposition \ref{prop:bounds}}

The bound $\delta^n \leq \frac{1}{1-\gamma}$ is clearly true for $n = 1$, since $\alpha_0 \leq 1$. Suppose now that $\delta^{n-1} \leq \frac{1}{1-\gamma}$ and $\lambda^{n-1} \leq \frac{1}{\gamma\left(1-\gamma\right)}$ for $n > 1$. Then, using the definition of $\delta^n$, we obtain
\begin{equation*}
\delta^n \leq \alpha_{n-1} + \left(1-\left(1-\gamma\right)\alpha_{n-1}\right)\frac{1}{1-\gamma} = \frac{1}{1-\gamma}\mbox{.}
\end{equation*}
Similarly, we can write
\begin{eqnarray*}
\lambda^n &\leq& \alpha^2_{n-1} + \left(1-\left(1-\gamma\right)\alpha_{n-1}\right)^2\frac{1}{\gamma\left(1-\gamma\right)}\\
&=& \frac{1}{\gamma\left(1-\gamma\right)} + \left(1+\frac{1-\gamma}{\gamma}\right)\alpha^2_{n-1} - \frac{2}{\gamma}\alpha_{n-1}\\
&=& \frac{1}{\gamma\left(1-\gamma\right)} + \frac{1}{\gamma}\alpha^2_{n-1} - \frac{2}{\gamma}\alpha_{n-1}\\
&\leq& \frac{1}{\gamma\left(1-\gamma\right)}\mbox{,}
\end{eqnarray*}
as required.

\subsection{Proof of Proposition \ref{prop:convex}}

Let $f\left(\alpha_{n-1}\right)$ be the right-hand side of (\ref{eq:firstcov}). First, observe that
\begin{equation*}
\frac{d^2f}{d\alpha^2_{n-1}} = 2\Exp\left[\left(\vbar^{n-1}-\Exp \vhat^n\right)^2\right] + 2\Exp\left[\left(\vhat^n - \Exp \vhat^n\right)^2\right] - 4Cov\left(\vbar^{n-1},\vhat^n\right)\mbox{.}
\end{equation*}
It is enough to show that
\begin{equation*}
2Cov\left(\vbar^{n-1},\vhat^n\right) \leq \Exp\left[\left(\vbar^{n-1}-\Exp \vhat^n\right)^2\right] + \Exp\left[\left(\vhat^n - \Exp \vhat^n\right)^2\right]\mbox{.}
\end{equation*}
Recall from (\ref{eq:biasvar}) and (\ref{eq:cov}) that
\begin{eqnarray*}
\Exp\left[\left(\vbar^{n-1}-\Exp \vhat^n\right)^2\right] &=& Var\left(\vbar^{n-1}\right) + \left(\Exp\vbar^{n-1}  -\Exp \vhat^n\right)^2,\\
Cov\left(\vbar^{n-1},\vhat^n\right) &=& \gamma Var\left(\vbar^{n-1}\right)\mbox{.}
\end{eqnarray*}
Observe that
\begin{equation*}
\Exp\left[\left(\vhat^n - \Exp \vhat^n\right)^2\right] = Var\left(\vhat^n\right) = \sigma^2 + \gamma^2 Var\left(\vbar^{n-1}\right)
\end{equation*}
and
\begin{equation*}
2\gamma Var\left(\vbar^{n-1}\right) \leq \left(1+\gamma^2\right)Var\left(\vbar^{n-1}\right)
\end{equation*}
since $\gamma^2 - 2\gamma + 1 = \left(\gamma - 1\right)^2 \geq 0$ and $Var\left(\vbar^{n-1}\right) \geq 0$ also. This completes the proof.

\subsection{Proof of Proposition \ref{prop:lowbound}}

We use an inductive argument to show that
\begin{equation*}
\lambda^{n-1} \geq \frac{1}{n-1} \quad \Rightarrow \quad \alpha_{n-1} \geq \frac{1-\gamma}{n} \quad \Rightarrow \quad \lambda^n \geq \frac{1}{n}
\end{equation*}
for all $n > 1$. Assuming $\alpha_0 = 1$, we have $\lambda^1 = 1$ by definition. Suppose now that $\lambda^{n-1} \geq \frac{1}{n-1}$ for some $n > 1$. We rewrite (\ref{eq:opt}) as
\begin{eqnarray*}
\alpha_{n-1} &=& \frac{1-\gamma}{n} - \frac{1-\gamma}{n} + \frac{\left(1-\gamma\right)\lambda^{n-1} \sigma^2 + \left(1 - \left(1 - \gamma\right)\delta^{n-1}\right)^2 c^2}{\left(1-\gamma\right)^2\lambda^{n-1} \sigma^2 + \left(1 - \left(1 - \gamma\right)\delta^{n-1}\right)^2 c^2 + \sigma^2}\\
&=& \frac{1-\gamma}{n} + A^{n-1}\mbox{,}
\end{eqnarray*}
where
\begin{equation*}
A^{n-1} = \frac{n\left(1-\gamma\right)\lambda^{n-1} \sigma^2 + \left(n-\left(1-\gamma\right)\right)\left(1 - \left(1 - \gamma\right)\delta^{n-1}\right)^2 c^2 - \left(1-\gamma\right)^3\lambda^{n-1}\sigma^2 - \left(1-\gamma\right)\sigma^2}{n\left(1-\gamma\right)^2\lambda^{n-1} \sigma^2 + n\left(1 - \left(1 - \gamma\right)\delta^{n-1}\right)^2 c^2 + n\sigma^2}\mbox{.}
\end{equation*}
The denominator of $A^{n-1}$ is clearly positive. To show that the numerator is positive as well, it suffices to show that
\begin{equation*}
n\left(1-\gamma\right)\lambda^{n-1} \sigma^2 \geq \left(1-\gamma\right)^3\lambda^{n-1}\sigma^2 + \left(1-\gamma\right)\sigma^2\mbox{.}
\end{equation*}
Because $1- \gamma \geq \left(1-\gamma\right)^3$ for $\gamma \in \left(0,1\right)$, it remains to show that
\begin{equation*}
\left(n-1\right)\left(1-\gamma\right)\lambda^{n-1}\sigma^2 \geq \left(1-\gamma\right)\sigma^2\mbox{,}
\end{equation*}
but this holds because $\lambda^{n-1} \geq \frac{1}{n-1}$ by the inductive hypothesis. Thus, $\alpha_{n-1} \geq \frac{1-\gamma}{n}$.

Using the result that $\alpha_{n-1} \geq \frac{1-\gamma}{n}$, we show that $\lambda^n \geq \frac{1}{n}$. Let us write
\begin{eqnarray*}
\lambda^{n-1} &=& \frac{1}{n-1} + L^{n-1}\mbox{,}\\
\alpha_{n-1} &=& \frac{1-\gamma}{n} + M^{n-1}\mbox{,}
\end{eqnarray*}
where $L^{n-1},M^{n-1} \geq 0$. Substituting these expressions into the definition of $\lambda^n$, we obtain
\begin{equation*}
\lambda^n = \left(\frac{1-\gamma}{n}+M^{n-1}\right)^2 + \left(1-\left(1-\gamma\right)\left(\frac{1-\gamma}{n}+M^{n-1}\right)\right)^2 \left(\frac{1}{n-1}+L^{n-1}\right)\mbox{.}
\end{equation*}
This expression can be rewritten as
\begin{eqnarray*}
\lambda^n &=& \frac{\left(1-\gamma\right)^2}{n^2} + 2\frac{1-\gamma}{n}M^{n-1} + \left(M^{n-1}\right)^2\\
&\,& + \left(\frac{n-\left(1-\gamma\right)^2}{n} - \left(1-\gamma\right)M^{n-1}\right)^2 \frac{1}{n-1} + L^{n-1}G^{n-1}\mbox{,}
\end{eqnarray*}
where $G^{n-1} \geq 0$. Since $n-1 \leq n-\left(1-\gamma\right)^2$, it follows that $\frac{1}{n-1} \geq \frac{1}{n-\left(1-\gamma\right)^2}$ and
\begin{eqnarray*}
\lambda^n &\geq& \frac{1}{n} + \frac{n}{n-\left(1-\gamma\right)^2}\left(M^{n-1}\right)^2 + L^{n-1}G^{n-1}\mbox{,}
\end{eqnarray*}
whence $\lambda^n \geq \frac{1}{n}$, as required.

\subsection{Proof of Proposition \ref{prop:biasgoestozero}}

We first show that the sequence $\left(\delta^n\right)^{\infty}_{n=1}$ is increasing. Recall from Proposition \ref{prop:bounds} that $\left(1-\gamma\right)\delta^{n-1} \leq 1$. It follows that $\left(1-\gamma\right)\alpha_{n-1}\delta^{n-1} + \delta^{n-1}\leq \alpha_{n-1} + \delta^{n-1}$, whence
\begin{equation*}
\delta^{n-1} \leq \alpha_{n-1} + \left(1-\left(1-\gamma\right)\alpha_{n-1}\right)\delta^{n-1} = \delta^n\mbox{,}
\end{equation*}
which means that $\left(\delta^n\right)$ is increasing. Since this sequence is also bounded by Proposition \ref{prop:bounds}, it has a limit $\delta^* \leq \frac{1}{1-\gamma}$.

Suppose now that $\delta^* < \frac{1}{1-\gamma}$. We rewrite the definition of $\delta^n$ as
\begin{equation*}
\left(\delta^n-\delta^{n-1}\right) + \delta^{n-1} = \alpha_{n-1} + \left(1-\left(1-\gamma\right)\alpha_{n-1}\right) \delta^{n-1}\mbox{.}
\end{equation*}
Subtracting $\delta^{n-1}$ from both sides yields
\begin{equation*}
\left(\delta^n-\delta^{n-1}\right) = \alpha_{n-1}\left(1- \left(1-\gamma\right)\delta^{n-1}\right)\mbox{.}
\end{equation*}
The left-hand side converges to zero as $n\rightarrow\infty$. On the right-hand side, if $\delta^* < \frac{1}{1-\gamma}$, then $1-\left(1-\gamma\right)\delta^{n-1} \rightarrow 1- \left(1-\gamma\right)\delta^* > 0$. It then follows that $\alpha_{n-1} \rightarrow 0$. However, we can see from (\ref{eq:opt}) that this is impossible if $\delta^* < \frac{1}{1-\gamma}$ because, in the limit, both the numerator and denominator will contain the strictly positive term $\left(1-\left(1-\gamma\right)\delta^*\right)^2$. All other terms in both the numerator and denominator of (\ref{eq:opt}) are positive. Therefore, it must be the case that $\delta^* = \frac{1}{1-\gamma}$.

\section{Discussion of OSAVI vs. BAKF}

In this section, we provide additional background for our approach and discuss its relation to the BAKF rule of \cite{GePo06}. This rule was originally presented under the name of OSA (Optimal Stepsize Algorithm). However, because it is not optimal for dynamic programming, we will refer to it by the alternate name of ``bias-adjusted Kalman filter'' given in \cite{Po11}. The BAKF rule is designed for a signal processing problem, in which there is a sequence of independent observations $\hat{X}^n$ with unknown means $\theta^n$ and common variance $\sigma^2$. The unknown means are estimated by the usual exponential smoothing technique
\begin{equation*}
\bar{\theta}^n\left(\alpha_{n-1}\right) = \left(1-\alpha_{n-1}\right)\bar{\theta}^{n-1} + \alpha_{n-1}\hat{X}^n\mbox{.}
\end{equation*}
To compute $\bar{\theta}^n$, the $n$th approximation, the BAKF rule chooses $\alpha_{n-1}$ to minimize
\begin{equation*}
\min_{0\leq\alpha_{n-1} \leq 1} \Exp\left[\left(\bar{\theta}^n\left(\alpha_{n-1}\right) - \theta^n\right)^2\right]\mbox{.}
\end{equation*}
The solution to this problem is given explicitly by the formula
\begin{equation*}
\alpha_{n-1} = 1 - \frac{\sigma^2}{\left(1 + \zeta^{n-1}\right)\sigma^2 + \left(\beta^n\right)^2}
\end{equation*}
where $\zeta^{n-1}$ is given by the recursive formula
\begin{eqnarray*}
\zeta^n = \left\{
\begin{array}{l l}
  \alpha^2_0 & n = 1\\
  \alpha^2_{n-1} + \left(1-\alpha_{n-1}\right)^2\zeta^{n-1} & n > 1
\end{array}
\right.
\end{eqnarray*}
and $\beta^n = \theta^n - \Exp\bar{\theta}^{n-1}$ is the bias in the smoothed estimate from the previous iteration.

The BAKF rule is particularly relevant to our study because it also chooses the stepsize to minimize the expected squared error of each prediction. For both BAKF and OSAVI, the prediction error is the squared difference between the mean of the new observation and the new estimate. In both cases, the resulting optimal stepsize contains one term representing the bias of the approximation, and one term representing the variance.

The crucial difference is as follows. BAKF is designed for a general signal processing problem in which the goal is to track a scalar moving signal. The work by \cite{GePo06} applies the computational formula of BAKF to an application in ADP, but in fact the derivation of BAKF uses a more general setting, whose main assumptions are violated in ADP. First, BAKF assumes that the observations used in smoothing are independent, which is not the case in approximate value iteration. Rather, the guiding principle of approximate value iteration is to bootstrap new observations from old approximations ($\vhat^n$ and $\vbar^{n-1}$ in the single-state, single-action model), due to the impossibility of obtaining unbiased estimates of the unknown value function.

By contrast, OSAVI makes the additional modeling assumption that observations are constructed according to (\ref{eq:vhat}), and thus the prediction error in (\ref{eq:obj}) is recast into a form that reflects the specific structure of ADP. This can be viewed as a special case of the BAKF derivation, but the additional structure imposed on the problem provides two important improvements over BAKF. First, OSAVI explicitly incorporates the dependence between the new observation and the old approximation. This dependence is crucial to the updating structure of ADP, but is not handled by BAKF. Second, the bias term $\beta^n$ in BAKF is unknown in practice (if we knew the bias, we would also know the true value). The work by \cite{GePo06} advocates using sample-based approximations of this quantity, giving rise to a second non-stationary estimation problem. On the other hand, the special structure assumed by OSAVI allows us to derive a closed-form expression for the bias, given by $\left(1-\left(1-\gamma\right)\delta^{n-1}\right)c^2$. In a general MDP, we also need to approximate $c$ (see Section \ref{sec:unknownc}). However, if we interpret $c$ as the average one-period reward earned by following an optimal policy in steady-state, this quantity is stationary, and thus is easier to estimate than the bias. We also see in Section \ref{sec:mainexp} that the secondary estimation procedure is less sensitive to the secondary stepsize $\nu_{n-1}$ for OSAVI than for BAKF.

We conclude our discussion by noting that, on one hand, BAKF provides more generality, and may be more appropriate in a general signal processing problem. On the other hand, in the specific context of approximate value iteration, where observations are constructed via bootstrapping and thus are inherently biased and dependent, OSAVI captures more of the specific structure of ADP, leading to improved performance.


\section*{Acknowledgment}

The authors are grateful to the Topology Atlas forum for several very helpful discussions. 
This research was supported in part by AFOSR contracts FA9550-08-1-0195, FA9550-11-1-0083 and FA9550-12-1-0200,
NSF contracts CMMI-1254298, IIS-142251 and IIS-1247696,
and ONR contract N00014-07-1-0150 through the Center for Dynamic Data Analysis.

\bibliographystyle{IEEEtran}
\bibliography{stepsize17_arxiv}

\end{document}